\documentclass{article}

\usepackage{graphicx} 
\usepackage{mathtools}
\usepackage{bm}
\usepackage{amssymb}
\usepackage{amsthm}
\usepackage{amsmath}
\usepackage{todonotes}
\usepackage[left=25mm,right=25mm,top=30mm,bottom=30mm]{geometry}
\usepackage{cases}
\usepackage{stmaryrd}
\usepackage[svgnames]{xcolor}
\usepackage{url}
\usepackage{algorithm}
\usepackage{algorithmicx}
\usepackage{algpseudocode}
\usepackage{placeins}
\usepackage{float}

\newtheorem{theorem}{Theorem}[section]

\theoremstyle{remark}
\newtheorem{remark}[theorem]{Remark}

\theoremstyle{definition}

\numberwithin{equation}{section}

\allowdisplaybreaks

\title{A stabilized dual-SAV parametric finite element framework for constrained planar geometric flows with mesh regularization}
\makeatletter
\renewcommand\@date{{%
  \vspace{-\baselineskip}%
  \large\centering
  \begin{tabular}{@{}c@{}}
    Koya Sakakibara\textsuperscript{1,2} \\
    \normalsize ksakaki@se.kanazawa-u.ac.jp
  \end{tabular}

  \bigskip

  \textsuperscript{1}Faculty of Mathematics and Physics, Institute of Science and Engineering, Kanazawa University\par
  \textsuperscript{2}RIKEN iTHEMS
}}

\begin{document}

\maketitle

\begin{abstract}
Parametric finite element discretizations of constrained geometric flows must
simultaneously address high-order geometric stiffness, mesh degeneration, and
nonlinear global constraints. This paper develops a stabilized dual-SAV
(scalar auxiliary variable) parametric finite element framework for planar
closed curves. The proposed formulation introduces separate auxiliary variables
for the physical geometric energy and for an artificial mesh regularization
energy. The mesh regularization is coupled only to tangential motion by
projecting out its normal variation, so that mesh redistribution changes the
parametrization without introducing an artificial normal driving force.
Based on this dual-energy structure, we construct a semi-implicit
frozen-metric scheme with zero-order stabilization. The scheme leads to linear
spatial response problems and satisfies discrete dissipation estimates for the
modified geometric and mesh SAV energies. Nonlinear global constraints are
handled by an algebraic block reduction: after solving a small number of
symmetric positive-definite response problems, the remaining nonlinear system
involves only the geometric auxiliary variable and the Lagrange multipliers.
For $K$ global constraints, this reduced nonlinear system has dimension
$K+1$; in particular, simultaneous area and length constraints lead to a
three-dimensional nonlinear system, independently of the number of mesh
vertices.
Numerical experiments for curve shortening, area-preserving curve shortening,
curve diffusion, and Helfrich-type flows illustrate the modified-energy
dissipation, the enforcement of geometric constraints, and the improvement of
mesh quality for both second- and fourth-order examples.
\end{abstract}

\section{Introduction}
\label{sec:intro}

\subsection{Background and numerical challenges}

Geometric evolution equations and moving boundary problems arise in a wide
range of models in materials science, fluid interfaces, and mathematical
biology. Typical examples include the curve shortening flow for the evolution
of plane curves \cite{epstein1987curve}, curvature-driven grain boundary motion
\cite{mullins1956two-dimensional}, surface-diffusion-type models for thermal
grooving and solid-state morphology changes \cite{mullins1957theory}, and
elastic flows associated with the Willmore and Helfrich energies
\cite{willmore1993riemannian,helfrich1973elastic}. Many of these problems can
be interpreted as gradient flows of geometric energies with respect to
appropriate metrics, such as the $L^2$ or $H^{-1}$ metric. In applications to
closed interfaces and biomembranes, the evolution is often subject to global
geometric constraints, such as the preservation of enclosed area, volume, or
perimeter \cite{elliott2010modeling,barrett2020parametric}.

From the numerical point of view, constrained geometric flows present several
closely related difficulties. First, the governing equations are nonlinear and
may contain high-order geometric derivatives; fourth-order flows, for example,
lead to severe time-step restrictions for explicit schemes. Second, global
constraints introduce Lagrange multipliers and nonlinear algebraic coupling at
the discrete level. Third, in parametric discretizations, the mesh points move
with the evolving interface and may cluster or spread unevenly, which can lead
to mesh degeneration and ill-conditioned algebraic systems.

Parametric finite element methods are attractive for such problems because
they represent the interface sharply and allow geometric quantities and global
constraints to be treated directly. Classical parametric finite element
approaches for geometric evolution equations go back to Dziuk's work on
evolving surfaces \cite{dziuk1991algorithm}, and the computation of geometric
partial differential equations and mean curvature flow has been extensively
surveyed by Deckelnick, Dziuk, and Elliott
\cite{deckelnick2005computation}. Parametric finite element methods have also
been developed for higher-order curvature flows and evolving hypersurfaces
\cite{barrett2007parametric,barrett2008parametric}, with good mesh properties
and stability mechanisms reviewed in \cite{barrett2020parametric}.
Nevertheless, combining dynamic mesh regularization, energy-stable
semi-implicit time stepping, and nonlinear global constraints in a single
linear large-scale framework remains a challenging problem.

\subsection{Mesh redistribution and structure-preserving parametric methods}

A central issue in parametric methods for geometric flows is the control of the
mesh parametrization. Since the normal velocity determines the geometric shape
of the curve, tangential velocities may be used to redistribute mesh points
without changing the underlying geometric evolution. This idea has been
developed in boundary tracking and front-tracking methods, where tangential
velocities are prescribed to improve the distribution of marker points along
the evolving curve \cite{kimura1997numerical,mikula2001evolution,mikula2004direct}.
Curvature-adjusted tangential velocities were further studied by
\v{S}ev\v{c}ovi\v{c} and Yazaki \cite{sevcovic2011evolution}, providing a
flexible mechanism for adaptive redistribution along plane curves.

Within the parametric finite element framework, mesh redistribution can also
be generated by the variational structure of the discretization. A prominent
example is the formulation of Barrett, Garcke, and N\"urnberg and its
extensions, which yield stable parametric finite element approximations with
favorable mesh properties for second- and fourth-order curvature-driven flows
\cite{barrett2007parametric,barrett2008parametric,barrett2020parametric}.
Another approach uses reparametrization via the DeTurck trick and has been
applied to curve shortening and mean curvature flow with improved mesh behavior
\cite{elliott2017approximations}. More recently, artificial tangential motion
has been incorporated into evolving finite element methods with convergence
analysis \cite{bai2024convergent}, while adaptive redistribution strategies
have been proposed for high-order curve flows \cite{duan2026adaptive}.

These developments show that tangential motion is an effective tool for
maintaining mesh quality. However, its incorporation into structure-preserving
schemes with nonlinear global constraints remains delicate. In our previous
work with Kemmochi and Miyatake, artificial tangential velocities were combined
with discrete gradient methods for planar curve flows
\cite{kemmochi2025structure}. That method preserves both energy dissipation
and geometric constraints at the level of the original discrete gradient
structure. However, it is fully implicit and leads to nonlinear algebraic
systems involving the spatial degrees of freedom.

Recent structure-preserving parametric finite element methods based on
intrinsic weak formulations or Lagrange multiplier approaches can also maintain
good mesh quality without explicitly prescribing an artificial mesh energy
\cite{garcke2025structure,bao2025energy}. These methods provide powerful
geometric discretizations. Nevertheless, a framework that combines explicit
mesh regularization, semi-implicit energy stabilization, and nonlinear global
constraints in an algebraically decoupled form has not yet been fully
developed.

\subsection{Energy stabilization schemes and the SAV approach}

Another important issue in the numerical approximation of geometric gradient
flows is the stable treatment of nonlinear energy variations. Fully implicit
energy-stable schemes are often robust, but they typically require the solution
of nonlinear systems at each time step. On the other hand, explicit or
semi-explicit treatments of nonlinear geometric forces may suffer from severe
time-step restrictions, especially for fourth-order flows. This difficulty has
motivated the development of energy-stabilized time discretizations, including
convex splitting methods \cite{eyre1998unconditionally} and invariant energy
quadratization methods \cite{yang2017numerical}.

The scalar auxiliary variable (SAV) approach introduced by Shen, Xu, and Yang
\cite{shen2018scalar} provides a particularly flexible strategy for gradient
flows. By introducing an auxiliary scalar variable associated with the nonlinear
part of the energy, the SAV formulation leads to linear schemes while retaining
a modified energy dissipation structure. Several extensions and refinements of
the SAV methodology have since been proposed, including multiple-SAV methods
for constrained gradient flows \cite{cheng2018multiple}, efficient SAV variants
\cite{huang2020highly}, SAV formulations for conservative and dissipative
systems \cite{kemmochi2022scalar}, and recent constant or weighted SAV
formulations designed to improve the relation between the modified and original
energies \cite{zhang2024new,huang2026weighted}.

Applying SAV ideas to parametric finite element discretizations of constrained
geometric flows, however, raises additional difficulties that do not appear in
standard Eulerian gradient-flow settings. The physical geometric energy is
invariant under tangential reparametrization and drives the normal evolution,
whereas the mesh regularization energy is artificial and depends on the chosen
parametrization. Moreover, global geometric constraints introduce Lagrange
multipliers that couple nonlinearly to the updated curve. If these components
are treated in a monolithic way, the resulting scheme may again require a large
nonlinear saddle-point solve. The goal of the present work is therefore to
combine the SAV stabilization principle with a parametric finite element
formulation in such a way that the physical geometric flow, the tangential mesh
regularization, and the global constraints remain algebraically separable.

\subsection{Proposed method: a stabilized dual-SAV framework}

In this paper, we develop a stabilized dual-SAV parametric finite element
framework for constrained geometric flows of planar closed curves. The main
idea is to represent the physical geometric energy and the artificial mesh
regularization energy by two separate scalar auxiliary variables. The physical
geometric energy drives the normal evolution of the curve, whereas the mesh
regularization energy is used only to generate a tangential redistribution of
mesh points. This separation allows the artificial mesh mechanism to improve
the parametrization without introducing an additional normal driving force into
the prescribed geometric evolution.

The proposed method is designed to combine three numerical requirements that
are difficult to satisfy simultaneously: stable treatment of high-order
geometric stiffness, active control of mesh quality, and accurate enforcement
of nonlinear global constraints. To this end, we introduce a frozen-metric
semi-implicit discretization with zero-order stabilization. The resulting
large-scale spatial problems are linear once the scalar auxiliary variables and
the Lagrange multipliers are fixed. We then exploit this structure to derive an
algebraic block reduction, in which the remaining nonlinear system involves
only the geometric scalar auxiliary variable and the Lagrange multipliers.

Thus, the distinctive point of the present framework is not the use of a
tangential velocity, an SAV reformulation, or a constraint treatment in
isolation. Rather, the contribution is to combine these ingredients in a
single parametric finite element framework in which the physical geometric
energy and the artificial mesh energy are separated, the mesh force is coupled
only to tangential motion, and the nonlinear constraint enforcement is reduced
to a low-dimensional algebraic system.

The main contributions of this work are summarized as follows.

\begin{enumerate}
    \item \textbf{Dual-SAV formulation with tangential mesh regularization.}
    We introduce independent scalar auxiliary variables for the physical
    geometric energy and the artificial mesh regularization energy. The mesh
    regularization force is projected onto the tangential direction, so that it
    changes the parametrization of the curve without modifying the prescribed
    normal geometric evolution at the continuous level.

    \item \textbf{Frozen-metric semi-implicit discretization.}
    We construct a semi-implicit time discretization in which all
    geometry-dependent metric, stabilization, and force terms are evaluated on
    the known curve. Combined with zero-order stabilization, this frozen-metric
    convention yields linear spatial response problems and discrete dissipation
    estimates for the modified geometric and mesh SAV energies.

    \item \textbf{Fully discrete parametric finite element formulation.}
    We formulate the method at the fully discrete level using a parametric
    finite element approximation of the evolving curve. The fully discrete
    system preserves the same algebraic cancellation structure as the
    time-discrete formulation and satisfies corresponding modified-energy
    dissipation estimates.

    \item \textbf{Algebraic block reduction for global constraints.}
    We show that the constrained fully discrete system can be decomposed into a
    small number of symmetric positive-definite spatial response problems and a
    reduced nonlinear algebraic system for the scalar auxiliary variable and
    the Lagrange multipliers. For $K$ global constraints, the nonlinear part
    has dimension $K+1$. In particular, simultaneous area and length
    constraints lead to a three-dimensional nonlinear system, independently of
    the number of mesh vertices.

    \item \textbf{Numerical validation for representative planar curve flows.}
    We test the framework on second- and fourth-order geometric flows,
    including curve shortening, area-preserving curve shortening, curve diffusion,
    and Helfrich-type flows. The experiments illustrate modified-energy
    dissipation, enforcement of geometric constraints, and improved mesh
    quality under tangential redistribution.
\end{enumerate}

To place the present method in context, Table~\ref{tab:intro_method_comparison}
compares the main structural features of several related classes of parametric
and structure-preserving methods. The aim is not to rank these methods, but to
identify the particular combination of properties targeted in this work.

\begin{table}[H]
\centering
\small
\caption{Qualitative comparison of related approaches and the present dual-SAV
PFEM. The comparison focuses on mesh control, energy structure, global
constraints, and the dimension of the nonlinear algebraic part.}
\label{tab:intro_method_comparison}
\begin{tabular}{@{}p{0.12\textwidth}p{0.20\textwidth}p{0.18\textwidth}p{0.19\textwidth}p{0.16\textwidth}@{}}
\hline
Approach
&
Mesh control
&
Energy structure
&
Global constraints
&
Nonlinear algebraic part
\\
\hline
BGN-type parametric FEMs
&
Good mesh properties are obtained through intrinsic weak formulations.
&
Energy-stable formulations are available for several geometric flows.
&
Can be incorporated, depending on the formulation.
&
Typically coupled to the spatial unknowns.
\\
\hline
Prescribed tangential velocity / DeTurck-type methods
&
Tangential motion is explicitly designed to improve parametrization.
&
Energy stability depends on the chosen time discretization.
&
Not the main focus in the basic formulations.
&
Usually not separated from the spatial evolution.
\\
\hline
Discrete-gradient methods with tangential velocities
&
Tangential redistribution can be combined with structure preservation.
&
Original discrete energy dissipation can be preserved.
&
Geometric constraints can be enforced accurately.
&
Fully implicit systems involving spatial degrees of freedom.
\\
\hline
SAV / IEQ-type stabilized schemes
&
Usually formulated in fixed or Eulerian variables; mesh quality is not the
central issue.
&
Modified-energy dissipation is obtained by auxiliary variables.
&
Constrained variants exist, but geometric constraints in parametric form
require additional treatment.
&
Often linear or linearly implicit after auxiliary-variable reformulation.
\\
\hline
Present dual-SAV PFEM
&
An artificial mesh energy drives only tangential redistribution, so its normal
variation is projected out.
&
Separate modified SAV energies are dissipated for the geometric and mesh
subsystems.
&
Nonlinear global constraints are imposed on the updated curve.
&
Reduced to a \((K+1)\)-dimensional nonlinear system after linear response
solves.
\\
\hline
\end{tabular}
\end{table}

The scope of the present paper is restricted to closed planar curves. Some
algebraic components of the method, such as the SAV splitting and the block
reduction of the constraint equations, are formally independent of the ambient
dimension. However, extending the full method to evolving surfaces in
$\mathbb{R}^3$ requires additional surface-mesh regularization mechanisms and a
separate stability and robustness analysis. We therefore regard the present
work as a planar-curve realization of the dual-SAV strategy and leave the
surface case to future work.

\subsection{Outline of the paper}

The remainder of this paper is organized as follows. In Section~\ref{sec:abstract_framework}, we introduce the geometric setting for planar closed curves and describe the abstract ingredients of the dual-energy formulation. Section~\ref{sec:continuous_dual_sav} presents the continuous dual-SAV formulation, including the distinction between the physical geometric auxiliary variable and the virtual mesh auxiliary variable. Section~\ref{sec:dissipation} proves the corresponding continuous modified-energy dissipation identities. In Section~\ref{sec:time_discretization}, we construct a semi-implicit frozen-metric time discretization with zero-order stabilization and establish a discrete SAV dissipation estimate. Section~\ref{sec:spatial_discretization} gives the parametric finite element discretization, proves the fully discrete modified-energy dissipation estimate, and derives the algebraic block reduction for nonlinear global constraints. Section~\ref{sec:numerical_validations} reports numerical experiments for representative constrained geometric flows of planar curves. Finally, Section~\ref{sec:conclusion} summarizes the main findings and discusses limitations and future extensions.

\section{Mathematical abstraction of dual-energy flows}
\label{sec:abstract_framework}

We first introduce the geometric notation and the abstract ingredients used in
the proposed dual-SAV formulation. The aim of this section is not to define a
new class of geometric evolution equations, but to isolate the variational,
metric, stabilization, and constraint structures that will be used in the
continuous formulation, the semi-implicit time discretization, and the fully
discrete parametric finite element scheme.

\subsection{Geometric setting and kinematics}

Let $\mathbb{S}^1:=\mathbb{R}/\mathbb{Z}$ be the reference periodic domain. We
consider a time-dependent closed parametric curve
\[
    \bm{\gamma}:\mathbb{S}^1\times[0,T)\to\mathbb{R}^2,
\]
where $\rho\in\mathbb{S}^1$ is a time-independent parameter and $t$ denotes
time. The prescribed initial curve is denoted by
$\bm{\gamma}_0:=\bm{\gamma}(\cdot,0)$.

For the continuous formulation, we assume that $\bm{\gamma}(\cdot,t)$ is a
regular smooth curve on the time interval under consideration. More precisely,
we assume that $|\partial_\rho\bm{\gamma}(\rho,t)|>0$ and that the curve has
sufficient smoothness for the geometric quantities appearing below to be
well-defined. For example, second-order flows require curvature-level
regularity, whereas fourth-order flows require higher regularity. This
smoothness assumption is used only for the continuous formulation; in the
fully discrete scheme, the curve is replaced by a polygonal parametric finite
element approximation.

The arc-length variable $s$ is defined by
$\mathrm{d}s=|\partial_\rho\bm{\gamma}|\,\mathrm{d}\rho$, and the unit tangent
vector is
\[
    \bm{\tau}
    =
    \partial_s\bm{\gamma}
    =
    \frac{\partial_\rho\bm{\gamma}}{|\partial_\rho\bm{\gamma}|}.
\]
The unit normal vector $\bm{\nu}$ is chosen so that
$(\bm{\tau},\bm{\nu})$ forms a positively oriented orthonormal basis of
$\mathbb{R}^2$. The scalar curvature $\kappa$ is defined by the Frenet--Serret
relations
\[
    \partial_s\bm{\tau}=\kappa\bm{\nu},
    \qquad
    \partial_s\bm{\nu}=-\kappa\bm{\tau}.
\]

Throughout the paper, $\mathbb{V}(\bm{\gamma})$ denotes a real scalar function space on
the evolving curve $\bm{\gamma}$, chosen according to the metric and
stabilization used in the particular flow under consideration. We assume that
$\mathbb{V}(\bm{\gamma})$ is sufficiently regular for all metric bilinear forms,
stabilization bilinear forms, and $L^2(\bm{\gamma})$ duality pairings appearing
below to be well-defined. For example, the pure $L^2$ metric can be formulated
on $L^2(\bm{\gamma})$, while Laplacian and biharmonic stabilization naturally
require $H^1(\bm{\gamma})$ and $H^2(\bm{\gamma})$ regularity, respectively.
For the $H^{-1}$ metric, the normal velocity is taken in the zero-mean class
\[
    \mathbb{V}_0(\bm{\gamma})
    :=
    \left\{
        v\in L^2(\bm{\gamma})
        \;:\;
        \int_{\bm{\gamma}} v\,\mathrm{d}s=0
    \right\},
\]
with the associated mean-zero potential formulation. In the fully discrete
formulation, this continuous space is replaced by the finite element space
$\mathbb{V}_h^n$ on the polygonal curve $\bm{\gamma}_h^n$.

The kinematics of the parametrized curve is described by the velocity field
$\partial_t\bm{\gamma}$. In the present framework, we decompose this velocity
into a normal geometric component and a tangential mesh component:
\[
    \partial_t\bm{\gamma}
    =
    V_\nu\bm{\nu}+V_\tau\bm{\tau}.
\]
Here $V_\nu\in\mathbb{V}(\bm{\gamma})$ determines the evolution of the geometric image
of the curve, whereas $V_\tau\in \mathbb{V}(\bm{\gamma})$ is an artificial tangential
velocity used to redistribute mesh points along the curve. Thus, $V_\tau$
changes the parametrization but not the geometric image of the curve.

\begin{remark}[No topological changes]
Throughout this paper, we assume that the evolving curve
$\bm{\gamma}(\cdot,t)$ does not undergo self-intersection, merging, pinching,
or any other topological change during the time interval under consideration.
Since the method is formulated in an explicit parametric representation, the
topology of the initial closed curve is fixed throughout the computation. Flows
involving topological singularities are outside the scope of the present work.
\end{remark}

\subsection{Abstract components of the dual-energy formulation}
\label{subsec:abstract_components}

We now specify the abstract components used to formulate the dual-SAV system.
These components describe the physical geometric evolution, the artificial mesh
redistribution, the global constraints, and the metric and stabilization
structures used in the subsequent discretization.

\begin{enumerate}
    \item \textbf{Geometric energy.}
    The physical shape evolution is driven by a geometric energy functional
    \(E_{\mathrm{geom}}(\bm{\gamma})\).
    Typical examples are the length energy, the Willmore bending energy, and
    Helfrich-type bending energies. We assume that
    $E_{\mathrm{geom}}$ is invariant under tangential reparametrization of the
    curve. Consequently, its first variation acts only in the normal direction.
    We write the corresponding scalar normal driving force as $w_{\mathrm{g}}$,
    so that the $L^2$-gradient has the form
    \[
        \nabla_{L^2} E_{\mathrm{geom}}
        =
        w_{\mathrm{g}}\bm{\nu}.
    \]

    \item \textbf{Mesh regularization energy.}
    To control the parametrization of the curve, we introduce an artificial
    mesh regularization energy \(E_{\mathrm{mesh}}(\bm{\gamma})\).
    In the examples below, this energy is taken as a weighted Dirichlet-type
    energy in the reference parameter:
    \[
        E_{\mathrm{mesh}}(\bm{\gamma})
        =
        \frac12
        \int_{\mathbb{S}^1}
        w(\rho,\bm{\gamma})
        |\partial_\rho \bm{\gamma}|^2\,d\rho,
        \qquad
        w(\rho,\bm{\gamma})>0 .
    \]
    Unlike the geometric energy, $E_{\mathrm{mesh}}$ depends on the
    parametrization. Its role is not to change the physical shape evolution,
    but to generate an artificial tangential redistribution of mesh points.
    We therefore use only the tangential component of its driving force and
    denote the corresponding scalar tangential force by $w_{\mathrm{m}}$.

    \item \textbf{Global geometric constraints.}
    We impose $K$ global geometric constraints
    \[
        C_i(\bm{\gamma})=0,
        \qquad i=1,\ldots,K .
    \]
    Examples include preservation of enclosed area or perimeter. Since these
    constraints are geometric, their first variations are normal. We write
    \[
        \nabla_{L^2} C_i
        =
        g_i\bm{\nu},
        \qquad i=1,\ldots,K,
    \]
    where $g_i$ denotes the scalar constraint gradient.

    \item \textbf{Metric bilinear forms.}
    The scalar driving forces are converted into velocities through
    state-dependent metric bilinear forms
    \[
        \mathcal{G}_\nu,
        \mathcal{G}_\tau
        :
        \mathbb{V}(\bm{\gamma})\times\mathbb{V}(\bm{\gamma})
        \to \mathbb{R}.
    \]
    The form $\mathcal{G}_\nu$ defines the metric for the physical normal
    gradient flow, while $\mathcal{G}_\tau$ defines the metric for the
    artificial tangential mesh redistribution. We assume that these forms are
    symmetric and positive definite on the relevant scalar function spaces.

    \item \textbf{Stabilization bilinear forms.}
    To stabilize the semi-implicit discretization, we introduce symmetric
    positive semi-definite bilinear forms
    \[
        \mathcal{D}_\nu,
        \mathcal{D}_\tau
        :
        \mathbb{V}(\bm{\gamma})\times\mathbb{V}(\bm{\gamma})
        \to \mathbb{R}.
    \]
    The normal stabilization $\mathcal{D}_\nu$ damps high-frequency modes
    associated with stiff geometric forces, while the tangential stabilization
    $\mathcal{D}_\tau$ damps oscillatory mesh redistribution. These forms are
    spatial stabilization operators. In the time-discrete scheme below, they
    are used in a zero-order-in-time manner, acting on the new velocity itself
    rather than on a velocity difference.
\end{enumerate}

\subsection{Examples of metric and stabilization bilinear forms}
\label{subsec:examples_metric_stabilization}

The abstract formulation above accommodates different physical metrics and
stabilization mechanisms within the same algebraic structure. We list here
typical choices at the continuous level. Their piecewise linear finite element
realizations will be specified later at the matrix level.

For the physical geometric flow, we mainly use the following two metrics on
scalar normal velocities.

\begin{itemize}
\item \textbf{$L^2$ metric.}
For curve shortening, Willmore-type, and Helfrich-type flows, we use the
standard $L^2$ inner product on the curve:
\[
    \mathcal{G}_{L^2,\bm{\gamma}}(U,V)
    :=
    (U,V)_{L^2(\bm{\gamma})}
    =
    \int_{\bm{\gamma}} UV\,\mathrm{d}s .
\]

\item \textbf{$H^{-1}$ metric.}
For the curve diffusion flow, the normal velocity is measured in the
$H^{-1}$ metric. This metric is defined on scalar fields with zero mean over
the curve. For
\[
    \int_{\bm{\gamma}} U\,\mathrm{d}s=0,
    \qquad
    \int_{\bm{\gamma}} V\,\mathrm{d}s=0,
\]
let $\psi_U$ and $\psi_V$ be the mean-zero solutions of
\[
    -\partial_s^2\psi_U=U,
    \qquad
    -\partial_s^2\psi_V=V
    \quad
    \text{on } \bm{\gamma}.
\]
Then we define
\[
    \mathcal{G}_{H^{-1},\bm{\gamma}}(U,V)
    :=
    \int_{\bm{\gamma}}
    \partial_s\psi_U\,\partial_s\psi_V\,\mathrm{d}s .
\]
The zero-mean restriction reflects the fact that the Laplace--Beltrami
operator on a closed curve has constants in its kernel. In the fully discrete
scheme, this restriction will be implemented on the corresponding discrete
zero-mean subspace.
\end{itemize}

For the tangential mesh regularization flow, we use the $L^2$ metric unless
otherwise stated:
\[
    \mathcal{G}_{\tau,\bm{\gamma}}(U,V)
    :=
    \int_{\bm{\gamma}} UV\,\mathrm{d}s .
\]

We next list typical stabilization forms. These terms are not part of the
original geometric gradient flow; rather, they are numerical damping mechanisms
used in the semi-implicit SAV discretization. The only structural assumption
needed in the dissipation analysis is that the stabilization form is symmetric
and positive semi-definite.

\begin{itemize}
\item \textbf{Laplacian stabilization.}
For second-order flows and for tangential mesh redistribution, we use
\[
    \mathcal{D}_{\mathrm{Lap},\bm{\gamma}}(U,V)
    :=
    \beta_2
    \int_{\bm{\gamma}}
    \partial_s U\,\partial_s V\,\mathrm{d}s,
    \qquad
    \beta_2\geq0 .
\]

\item \textbf{Biharmonic-type stabilization.}
For fourth-order geometric flows, we use the continuous prototype
\[
    \mathcal{D}_{\mathrm{bi},\bm{\gamma}}(U,V)
    :=
    \beta_4
    \int_{\bm{\gamma}}
    \partial_s^2 U\,\partial_s^2 V\,\mathrm{d}s,
    \qquad
    \beta_4\geq0 .
\]
This form is only a continuous prototype. Since the fully discrete scheme uses
piecewise linear finite elements, it is not implemented as a conforming
\(H^2\) bilinear form. The corresponding matrix realization is specified in
Section~\ref{sec:spatial_discretization}.

\item \textbf{Hybrid stabilization for nonlocal fourth-order flows.}
For nonlocal fourth-order flows, in particular for the curve diffusion flow,
we also allow the hybrid stabilization
\[
    \mathcal{D}_{\mathrm{hyb},\bm{\gamma}}(U,V)
    :=
    \beta_4
    \int_{\bm{\gamma}}
    \partial_s^2 U\,\partial_s^2 V\,\mathrm{d}s
    +
    \beta_2
    \int_{\bm{\gamma}}
    \partial_s U\,\partial_s V\,\mathrm{d}s,
    \qquad
    \beta_4,\beta_2\geq0 .
\]
The biharmonic component damps high-frequency modes associated with
fourth-order stiffness, while the Laplacian component provides additional
damping of intermediate-frequency oscillations. Since both terms are symmetric
and positive semi-definite, this hybrid choice fits the abstract assumption on
the stabilization form.
\end{itemize}

\begin{remark}[Separation of force, metric, and stabilization]
The thermodynamic force is determined by the geometric energy, while the metric
determines how this force is converted into a velocity. For example, the curve
shortening flow and the curve diffusion flow both use the length energy and
therefore have the same scalar geometric force
\[
    w_{\mathrm{g}}=-\kappa .
\]
The difference lies in the metric: the curve shortening flow uses the $L^2$
metric, whereas the curve diffusion flow uses the $H^{-1}$ metric. The
stabilization form is a further numerical component used to damp explicitly
treated stiffness; it does not change the algebraic cancellation mechanism as
long as it is used consistently and is symmetric positive semi-definite.
\end{remark}

\section{Unified continuous dual-SAV formulation}
\label{sec:continuous_dual_sav}

We now introduce the continuous dual-SAV formulation. As a reference point, we
first write the constrained gradient-flow system before introducing auxiliary
variables. Let
\(V_\nu,V_\tau\in\mathbb{V}(\bm{\gamma})\)
denote the scalar normal and tangential velocity components. The constrained
dual-energy flow seeks \(V_\nu\) and \(V_\tau\) such that, for all test functions
\(\varphi_\nu,\varphi_\tau\in\mathbb{V}(\bm{\gamma})\),
\[
    \mathcal{G}_\nu(V_\nu,\varphi_\nu)
    =
    -(w_{\mathrm{g}},\varphi_\nu)_{L^2(\bm{\gamma})}
    -
    \sum_{i=1}^K
    \lambda_i(t)(g_i,\varphi_\nu)_{L^2(\bm{\gamma})},
\]
and
\[
    \mathcal{G}_\tau(V_\tau,\varphi_\tau)
    =
    -(w_{\mathrm{m}},\varphi_\tau)_{L^2(\bm{\gamma})}.
\]
The global constraints are imposed through
\[
    \frac{\mathrm{d}}{\mathrm{d}t}C_i(\bm{\gamma})
    =
    (g_i,V_\nu)_{L^2(\bm{\gamma})}
    =
    0,
    \qquad
    i=1,\ldots,K.
\]
This formulation separates the normal geometric evolution from the tangential
mesh redistribution. A direct implicit discretization, however, would generally
lead to a nonlinear system involving geometry-dependent forces, metrics, and
Lagrange multipliers. Conversely, an explicit treatment of the nonlinear forces
may impose severe time-step restrictions, especially for fourth-order flows.
The role of the SAV reformulation below is to retain a linear large-scale
spatial structure while preserving a modified energy-dissipation mechanism.

\subsection{Independent scalar auxiliary variables}
\label{subsec:independent_sav_variables}

We introduce the scalar functionals
\[
    S_{\mathrm{g}}(\bm{\gamma})
    :=
    \sqrt{E_{\mathrm{geom}}(\bm{\gamma})+C_{\mathrm{g}}},
    \qquad
    S_{\mathrm{m}}(\bm{\gamma})
    :=
    \sqrt{E_{\mathrm{mesh}}(\bm{\gamma})+C_{\mathrm{m}}},
\]
where \(C_{\mathrm{g}}>0\) and \(C_{\mathrm{m}}>0\) are fixed constants chosen
so that the radicands are positive. We then introduce two time-dependent
scalar auxiliary variables \(r_{\mathrm{g}}(t)\) and \(r_{\mathrm{m}}(t)\).

The variable \(r_{\mathrm{g}}\) is the standard SAV associated with the physical
geometric energy. In contrast, \(r_{\mathrm{m}}\) is used only to represent the
tangential work of the mesh regularization energy. Thus, \(r_{\mathrm{m}}\)
should be interpreted as a virtual auxiliary variable for tangential mesh
redistribution, rather than as an exact tracker of
\(S_{\mathrm{m}}(\bm{\gamma}(t))\) along the full motion.

The auxiliary variables are defined by
\[
\begin{aligned}
    \frac{\mathrm{d}r_{\mathrm{g}}}{\mathrm{d}t}
    &=
    \frac{1}{2S_{\mathrm{g}}(\bm{\gamma})}
    (w_{\mathrm{g}},V_\nu)_{L^2(\bm{\gamma})},
    &
    r_{\mathrm{g}}(0)
    &=
    S_{\mathrm{g}}(\bm{\gamma}_0),
    \\
    \frac{\mathrm{d}r_{\mathrm{m}}}{\mathrm{d}t}
    &=
    \frac{1}{2S_{\mathrm{m}}(\bm{\gamma})}
    (w_{\mathrm{m}},V_\tau)_{L^2(\bm{\gamma})},
    &
    r_{\mathrm{m}}(0)
    &=
    S_{\mathrm{m}}(\bm{\gamma}_0).
\end{aligned}
\]
The first equation is consistent with the usual SAV representation of the
geometric energy, because \(E_{\mathrm{geom}}\) is invariant under tangential
reparametrization and its variation is determined by the normal velocity. The
second equation intentionally uses only the tangential partial work of the mesh
energy. This is the device that prevents the artificial mesh regularization
from contributing to the normal geometric evolution.

\begin{remark}[Virtual mesh auxiliary variable]
For the geometric energy, the chain rule gives
\[
    \frac{\mathrm{d}}{\mathrm{d}t}
    E_{\mathrm{geom}}(\bm{\gamma}(t))
    =
    (w_{\mathrm{g}},V_\nu)_{L^2(\bm{\gamma})}.
\]
Hence, at the continuous level, \(r_{\mathrm{g}}\) remains consistent with
\(S_{\mathrm{g}}(\bm{\gamma})\) as long as
\(r_{\mathrm{g}}(0)=S_{\mathrm{g}}(\bm{\gamma}_0)\). The situation is different
for the mesh regularization energy. Since \(E_{\mathrm{mesh}}\) depends on the
parametrization, its full time derivative would contain contributions from both
the normal and tangential components of the velocity. Including the normal
contribution would allow the artificial mesh energy to affect the physical shape
evolution. We therefore define \(r_{\mathrm{m}}\) through the tangential work
only. Consequently, \(r_{\mathrm{m}}\) is not intended to coincide with
\(S_{\mathrm{m}}(\bm{\gamma}(t))\) for the full motion; it is a virtual driving
variable whose role is to stabilize and scale the tangential mesh redistribution.
\end{remark}

\subsection{Modified continuous dual-SAV system}
\label{subsec:modified_continuous_dual_sav_system}

We now couple the auxiliary variables to the normal and tangential velocity
equations. The geometric and mesh driving forces are scaled by
\[
    \frac{r_{\mathrm{g}}(t)}{S_{\mathrm{g}}(\bm{\gamma})},
    \qquad
    \frac{r_{\mathrm{m}}(t)}{S_{\mathrm{m}}(\bm{\gamma})},
\]
respectively. These two scaling factors have different interpretations. Since
\(r_{\mathrm{g}}(t)\) is consistent with
\(S_{\mathrm{g}}(\bm{\gamma}(t))\) at the continuous level, the factor
\(r_{\mathrm{g}}(t)/S_{\mathrm{g}}(\bm{\gamma})\) is equal to one along smooth
solutions initialized by
\(r_{\mathrm{g}}(0)=S_{\mathrm{g}}(\bm{\gamma}_0)\). Therefore, the geometric
SAV reformulation does not change the normal geometric evolution at the
continuous level. In contrast, \(r_{\mathrm{m}}(t)\) is not intended to coincide
with \(S_{\mathrm{m}}(\bm{\gamma}(t))\) for the full motion. The factor
\(r_{\mathrm{m}}(t)/S_{\mathrm{m}}(\bm{\gamma})\) should instead be viewed as a
dynamic scaling of the tangential mesh relaxation.

To obtain an energy-cancellation structure in the presence of constraints, we
add to the evolution equation for \(r_{\mathrm{g}}\) the constraint-consistent
term
\[
    \sum_{i=1}^K
    \frac{\lambda_i(t)}{2r_{\mathrm{g}}(t)}
    (g_i,V_\nu)_{L^2(\bm{\gamma})}.
\]
This term vanishes for solutions satisfying the constraint dynamics, since
\[
    (g_i,V_\nu)_{L^2(\bm{\gamma})}
    =
    \frac{\mathrm{d}}{\mathrm{d}t}C_i(\bm{\gamma})
    =
    0 .
\]
Thus, it does not alter the continuous constrained flow. Its role is algebraic:
it creates the cancellation of the Lagrange multiplier terms between the
normal velocity equation and the SAV equation in the modified-energy identity.

The modified continuous dual-SAV system is as follows. Find
$\bm{\gamma}(t)$, scalar velocities
$V_\nu,V_\tau\in\mathbb{V}(\bm{\gamma})$, auxiliary variables
$r_{\mathrm{g}}(t),r_{\mathrm{m}}(t)\in\mathbb{R}$, and Lagrange multipliers
$\lambda_i(t)\in\mathbb{R}$, $i=1,\ldots,K$, such that, for all
$\varphi_\nu,\varphi_\tau\in\mathbb{V}(\bm{\gamma})$,
\begin{subequations}
\label{eq:continuous_dual_sav_system}
\begin{align}
    \mathcal{G}_\nu(V_\nu,\varphi_\nu)
    &=
    -
    \frac{r_{\mathrm{g}}(t)}{S_{\mathrm{g}}(\bm{\gamma})}
    (w_{\mathrm{g}},\varphi_\nu)_{L^2(\bm{\gamma})}
    -
    \sum_{i=1}^K
    \lambda_i(t)(g_i,\varphi_\nu)_{L^2(\bm{\gamma})},
    \label{eq:continuous_dual_sav_normal}
    \\
    \frac{\mathrm{d}r_{\mathrm{g}}}{\mathrm{d}t}(t)
    &=
    \frac{1}{2S_{\mathrm{g}}(\bm{\gamma})}
    (w_{\mathrm{g}},V_\nu)_{L^2(\bm{\gamma})}
    +
    \sum_{i=1}^K
    \frac{\lambda_i(t)}{2r_{\mathrm{g}}(t)}
    (g_i,V_\nu)_{L^2(\bm{\gamma})},
    \label{eq:continuous_dual_sav_rg}
    \\
    \mathcal{G}_\tau(V_\tau,\varphi_\tau)
    &=
    -
    \frac{r_{\mathrm{m}}(t)}{S_{\mathrm{m}}(\bm{\gamma})}
    (w_{\mathrm{m}},\varphi_\tau)_{L^2(\bm{\gamma})},
    \label{eq:continuous_dual_sav_tangent}
    \\
    \frac{\mathrm{d}r_{\mathrm{m}}}{\mathrm{d}t}(t)
    &=
    \frac{1}{2S_{\mathrm{m}}(\bm{\gamma})}
    (w_{\mathrm{m}},V_\tau)_{L^2(\bm{\gamma})},
    \label{eq:continuous_dual_sav_rm}
    \\
    \frac{\mathrm{d}}{\mathrm{d}t}C_i(\bm{\gamma})
    &=
    (g_i,V_\nu)_{L^2(\bm{\gamma})}
    =
    0,
    \qquad i=1,\ldots,K .
    \label{eq:continuous_dual_sav_constraints}
\end{align}
\end{subequations}
The initial conditions are
\[
    \bm{\gamma}(\cdot,0)=\bm{\gamma}_0,
    \qquad
    r_{\mathrm{g}}(0)=S_{\mathrm{g}}(\bm{\gamma}_0),
    \qquad
    r_{\mathrm{m}}(0)=S_{\mathrm{m}}(\bm{\gamma}_0).
\]

\begin{remark}[Constraint-consistent coupling]
The additional Lagrange multiplier term in
\eqref{eq:continuous_dual_sav_rg} is identically zero along the constrained
continuous flow. Nevertheless, it is essential for the modified-energy
identity. Indeed, testing \eqref{eq:continuous_dual_sav_normal} with
\(V_\nu\) and multiplying \eqref{eq:continuous_dual_sav_rg} by
\(2r_{\mathrm{g}}(t)\) gives an algebraic cancellation of the scaled geometric
force term and all Lagrange multiplier terms. This is the continuous
counterpart of the block structure used later in the fully discrete scheme.
\end{remark}

\section{Continuous SAV dissipation}
\label{sec:dissipation}

We next show that the continuous dual-SAV system admits two independent
dissipation identities. These identities are stated for the modified SAV energies
\(r_{\mathrm{g}}^2\) and \(r_{\mathrm{m}}^2\).
Thus, the result is a dissipation property of the auxiliary-variable
formulation. It is not a direct dissipation statement for the original energies
\(E_{\mathrm{geom}}\) and \(E_{\mathrm{mesh}}\).

\begin{theorem}[Continuous orthogonal SAV dissipation]
\label{thm:continuous_orthogonal_sav_dissipation}
Assume that the metric forms satisfy the positivity assumptions stated in
Section~\ref{subsec:abstract_components}. Suppose that the solution of the continuous dual-SAV system
\eqref{eq:continuous_dual_sav_system} is sufficiently smooth and that
\(r_{\mathrm{g}}(t)\neq0\) on the time interval under consideration. Then
\begin{align}
    \frac{\mathrm{d}}{\mathrm{d}t}r_{\mathrm{g}}^2(t)
    &=
    -
    \mathcal{G}_\nu(V_\nu,V_\nu)
    \leq 0,
    \label{eq:continuous_geometric_sav_dissipation}
    \\
    \frac{\mathrm{d}}{\mathrm{d}t}r_{\mathrm{m}}^2(t)
    &=
    -
    \mathcal{G}_\tau(V_\tau,V_\tau)
    \leq 0.
    \label{eq:continuous_mesh_sav_dissipation}
\end{align}
In particular, the modified geometric SAV energy and the modified mesh SAV
energy dissipate independently.
\end{theorem}

\begin{proof}
We prove the two identities separately.

First, we consider the normal geometric subsystem. Multiplying
\eqref{eq:continuous_dual_sav_rg} by $2r_{\mathrm{g}}(t)$ gives
\[
    \frac{\mathrm{d}}{\mathrm{d}t}r_{\mathrm{g}}^2(t)
    =
    \frac{r_{\mathrm{g}}(t)}{S_{\mathrm{g}}(\bm{\gamma})}
    (w_{\mathrm{g}},V_\nu)_{L^2(\bm{\gamma})}
    +
    \sum_{i=1}^K
    \lambda_i(t)
    (g_i,V_\nu)_{L^2(\bm{\gamma})}.
\]
On the other hand, choosing $\varphi_\nu=V_\nu$ in
\eqref{eq:continuous_dual_sav_normal} yields
\[
    \mathcal{G}_\nu(V_\nu,V_\nu)
    =
    -
    \frac{r_{\mathrm{g}}(t)}{S_{\mathrm{g}}(\bm{\gamma})}
    (w_{\mathrm{g}},V_\nu)_{L^2(\bm{\gamma})}
    -
    \sum_{i=1}^K
    \lambda_i(t)
    (g_i,V_\nu)_{L^2(\bm{\gamma})}.
\]
Adding these two relations, the scaled geometric force term and all Lagrange
multiplier terms cancel. Therefore,
\[
    \frac{\mathrm{d}}{\mathrm{d}t}r_{\mathrm{g}}^2(t)
    +
    \mathcal{G}_\nu(V_\nu,V_\nu)
    =
    0.
\]
Since $\mathcal{G}_\nu$ is positive definite on the relevant scalar function
space, we obtain
\[
    \frac{\mathrm{d}}{\mathrm{d}t}r_{\mathrm{g}}^2(t)
    =
    -
    \mathcal{G}_\nu(V_\nu,V_\nu)
    \leq 0.
\]

Next, we consider the tangential mesh subsystem. Multiplying
\eqref{eq:continuous_dual_sav_rm} by $2r_{\mathrm{m}}(t)$ gives
\[
    \frac{\mathrm{d}}{\mathrm{d}t}r_{\mathrm{m}}^2(t)
    =
    \frac{r_{\mathrm{m}}(t)}{S_{\mathrm{m}}(\bm{\gamma})}
    (w_{\mathrm{m}},V_\tau)_{L^2(\bm{\gamma})}.
\]
Choosing $\varphi_\tau=V_\tau$ in
\eqref{eq:continuous_dual_sav_tangent} gives
\[
    \mathcal{G}_\tau(V_\tau,V_\tau)
    =
    -
    \frac{r_{\mathrm{m}}(t)}{S_{\mathrm{m}}(\bm{\gamma})}
    (w_{\mathrm{m}},V_\tau)_{L^2(\bm{\gamma})}.
\]
Adding the two identities yields
\[
    \frac{\mathrm{d}}{\mathrm{d}t}r_{\mathrm{m}}^2(t)
    +
    \mathcal{G}_\tau(V_\tau,V_\tau)
    =
    0.
\]
Since $\mathcal{G}_\tau$ is positive definite, we conclude that
\[
    \frac{\mathrm{d}}{\mathrm{d}t}r_{\mathrm{m}}^2(t)
    =
    -
    \mathcal{G}_\tau(V_\tau,V_\tau)
    \leq 0.
\]
This completes the proof.
\end{proof}

\begin{remark}[Meaning of the orthogonal dissipation]
Theorem~\ref{thm:continuous_orthogonal_sav_dissipation} shows that the normal
geometric part and the tangential mesh part are dissipated through different
metric structures. The cancellation in the normal subsystem uses the
constraint-consistent coupling in the $r_{\mathrm{g}}$ equation, whereas the
tangential subsystem is independent of the Lagrange multipliers. Consequently,
the mesh redistribution can be stabilized through the tangential equation
without introducing an artificial normal force into the geometric evolution.
\end{remark}

\section{Semi-implicit temporal discretization}
\label{sec:time_discretization}

We now construct a semi-implicit time discretization of the continuous
dual-SAV system \eqref{eq:continuous_dual_sav_system}. The purpose of this
section is to identify the time-discrete algebraic structure that preserves
the SAV cancellation mechanism. The fully discrete parametric finite element
formulation and its corresponding dissipation estimate will be given in
Section~\ref{sec:spatial_discretization}.

The discretization is based on two ingredients. The first is a frozen-metric
convention: during one time step, all geometry-dependent inner products,
metric bilinear forms, stabilization bilinear forms, driving forces, and
constraint gradients are evaluated on the known curve at time level $n$. The
second is a zero-order-in-time use of the spatial stabilization forms
introduced in Section~\ref{sec:abstract_framework}. Namely, the stabilization
terms act on the new velocity itself, rather than on a difference of
velocities. This choice is essential for retaining the algebraic cancellation
used in the modified-energy estimate and avoids propagating velocity history
through the stabilization operator.

\subsection{Frozen metric and discrete notation}
\label{subsec:frozen_metric_discrete_notation}

Let \(\Delta t>0\) be the time step size and set
\(t^n:=n\Delta t\). The approximation of the curve at time \(t^n\) is denoted
by \(\bm{\gamma}^n\). Throughout the time step from \(t^n\) to \(t^{n+1}\),
all state-dependent quantities are evaluated on \(\bm{\gamma}^n\) and treated
as fixed.

For notational simplicity, we write
\[
    (u,v)^n
    :=
    (u,v)_{L^2(\bm{\gamma}^n)} .
\]
The frozen metric bilinear forms are denoted by
\(\mathcal{G}_\nu^n,\mathcal{G}_\tau^n\), and the frozen stabilization
bilinear forms by
\(\mathcal{D}_\nu^n,\mathcal{D}_\tau^n\). The geometric and mesh driving
forces evaluated on \(\bm{\gamma}^n\) are denoted by
\(w_{\mathrm{g}}^n,w_{\mathrm{m}}^n\), and the scalar constraint gradients by
\(g_i^n\), \(i=1,\ldots,K\). We also set
\[
    S_{\mathrm{g}}^n:=S_{\mathrm{g}}(\bm{\gamma}^n),
    \qquad
    S_{\mathrm{m}}^n:=S_{\mathrm{m}}(\bm{\gamma}^n).
\]

The scalar function space on the frozen curve is denoted by
\(\mathbb{V}^n:=\mathbb{V}(\bm{\gamma}^n)\). The unknown scalar normal and
tangential velocities at the next time step are sought in this frozen space:
\(V_\nu^{n+1},V_\tau^{n+1}\in\mathbb{V}^n\). After these velocities are
computed, the curve is updated by
\[
    \bm{\gamma}^{n+1}
    =
    \bm{\gamma}^{n}
    +
    \Delta t
    \left(
        V_\nu^{n+1}\bm{\nu}^{n}
        +
        V_\tau^{n+1}\bm{\tau}^{n}
    \right),
\]
where \(\bm{\nu}^{n}\) and \(\bm{\tau}^{n}\) are the unit normal and tangent
vectors associated with the frozen curve \(\bm{\gamma}^{n}\).

\subsection{Zero-order stabilized dual-SAV scheme}
\label{subsec:zero_order_stabilized_dual_sav_scheme}

Given \(\bm{\gamma}^n\), \(r_{\mathrm{g}}^n\), and \(r_{\mathrm{m}}^n\), we seek
\(V_\nu^{n+1},V_\tau^{n+1}\in\mathbb{V}^n\),
\(r_{\mathrm{g}}^{n+1},r_{\mathrm{m}}^{n+1}\in\mathbb{R}\), and
\(\lambda_i^{n+1}\in\mathbb{R}\), \(i=1,\ldots,K\), such that, for all test
functions \(\varphi_\nu,\varphi_\tau\in\mathbb{V}^n\),
\begin{subequations}
\label{eq:time_discrete_dual_sav_scheme}
\begin{align}
    \mathcal{G}_\nu^n
    (V_\nu^{n+1},\varphi_\nu)
    +
    \Delta t\,
    \mathcal{D}_\nu^n
    (V_\nu^{n+1},\varphi_\nu)
    &=
    -
    \frac{r_{\mathrm{g}}^{n+1}}{S_{\mathrm{g}}^n}
    (w_{\mathrm{g}}^n,\varphi_\nu)^n
    -
    \sum_{i=1}^K
    \lambda_i^{n+1}
    (g_i^n,\varphi_\nu)^n,
    \label{eq:time_discrete_normal}
    \\
    \frac{r_{\mathrm{g}}^{n+1}-r_{\mathrm{g}}^n}{\Delta t}
    &=
    \frac{1}{2S_{\mathrm{g}}^n}
    (w_{\mathrm{g}}^n,V_\nu^{n+1})^n
    +
    \sum_{i=1}^K
    \frac{\lambda_i^{n+1}}{2r_{\mathrm{g}}^{n+1}}
    (g_i^n,V_\nu^{n+1})^n,
    \label{eq:time_discrete_rg}
    \\
    \mathcal{G}_\tau^n
    (V_\tau^{n+1},\varphi_\tau)
    +
    \Delta t\,
    \mathcal{D}_\tau^n
    (V_\tau^{n+1},\varphi_\tau)
    &=
    -
    \frac{r_{\mathrm{m}}^{n+1}}{S_{\mathrm{m}}^n}
    (w_{\mathrm{m}}^n,\varphi_\tau)^n,
    \label{eq:time_discrete_tangent}
    \\
    \frac{r_{\mathrm{m}}^{n+1}-r_{\mathrm{m}}^n}{\Delta t}
    &=
    \frac{1}{2S_{\mathrm{m}}^n}
    (w_{\mathrm{m}}^n,V_\tau^{n+1})^n .
    \label{eq:time_discrete_rm}
\end{align}
The nonlinear global constraints are imposed on the updated curve:
\begin{equation}
    C_i\!\left(
        \bm{\gamma}^n
        +
        \Delta t
        \left(
            V_\nu^{n+1}\bm{\nu}^n
            +
            V_\tau^{n+1}\bm{\tau}^n
        \right)
    \right)
    =
    0,
    \qquad
    i=1,\ldots,K.
    \label{eq:time_discrete_constraints}
\end{equation}
\end{subequations}

The terms
\[
    \Delta t\,\mathcal{D}_\nu^n(V_\nu^{n+1},\varphi_\nu),
    \qquad
    \Delta t\,\mathcal{D}_\tau^n(V_\tau^{n+1},\varphi_\tau)
\]
are zero-order damping terms in time. They act on the new velocities
\(V_\nu^{n+1}\) and \(V_\tau^{n+1}\) themselves and do not involve velocity
differences such as \(V_\nu^{n+1}-V_\nu^n\) or
\(V_\tau^{n+1}-V_\tau^n\). Thus, no velocity history is propagated through the
stabilization operator. This choice is essential for the algebraic SAV
cancellation used in the dissipation estimate below.

\begin{remark}[Meaning of zero-order stabilization]
The term ``zero-order'' refers to the temporal use of the stabilization
operator, not to the spatial differential order of
$\mathcal{D}_\nu^n$ or $\mathcal{D}_\tau^n$. The bilinear forms may contain
first- or second-order spatial derivatives, such as Laplacian or biharmonic
stabilization. In the time-discrete scheme, however, they enter as absolute
damping terms applied to the new velocity. This is different from stabilization
strategies that propagate velocity history through terms involving
$V^{n+1}-V^n$.
\end{remark}

\subsection{Discrete SAV dissipation}

We now prove that the zero-order stabilized time-discrete scheme satisfies
discrete dissipation estimates for the modified SAV energies. The proof relies
on the frozen-metric convention and on the consistent use of the same frozen
inner products in the velocity equations and in the SAV update equations.

\begin{theorem}[Time-discrete modified-energy dissipation]
\label{thm:time_discrete_modified_energy_dissipation}
Under the assumptions on the metric and stabilization forms stated in
Sections~\ref{subsec:abstract_components} and
\ref{subsec:examples_metric_stabilization}, suppose that
$r_{\mathrm{g}}^{n+1}\neq 0$ and that
$
    V_\nu^{n+1}
$,
$
    V_\tau^{n+1}
$,
$
    r_{\mathrm{g}}^{n+1}
$,
$
    r_{\mathrm{m}}^{n+1}
$, 
$
    \lambda_1^{n+1},\ldots,\lambda_K^{n+1}
$
satisfy the time-discrete scheme
\eqref{eq:time_discrete_dual_sav_scheme}.
Then
\begin{align}
    \frac{
        (r_{\mathrm{g}}^{n+1})^2
        -
        (r_{\mathrm{g}}^n)^2
    }{\Delta t}
    &\leq
    -
    \mathcal{G}_\nu^n(V_\nu^{n+1},V_\nu^{n+1})
    -
    \Delta t\,
    \mathcal{D}_\nu^n(V_\nu^{n+1},V_\nu^{n+1})
    \leq 0,
    \label{eq:time_discrete_geometric_sav_dissipation}
    \\
    \frac{
        (r_{\mathrm{m}}^{n+1})^2
        -
        (r_{\mathrm{m}}^n)^2
    }{\Delta t}
    &\leq
    -
    \mathcal{G}_\tau^n(V_\tau^{n+1},V_\tau^{n+1})
    -
    \Delta t\,
    \mathcal{D}_\tau^n(V_\tau^{n+1},V_\tau^{n+1})
    \leq 0.
    \label{eq:time_discrete_mesh_sav_dissipation}
\end{align}
Consequently, the time-discrete modified geometric and mesh SAV energies are
nonincreasing for any time step size for which the discrete system is solvable.
\end{theorem}

\begin{proof}
We use the elementary identity
\[
    2a(a-b)=a^2-b^2+(a-b)^2 .
\]

We first consider the normal geometric subsystem. Taking
$\varphi_\nu=V_\nu^{n+1}$ in \eqref{eq:time_discrete_normal} gives
\begin{align}
    \mathcal{G}_\nu^n(V_\nu^{n+1},V_\nu^{n+1})
    +
    \Delta t\,
    \mathcal{D}_\nu^n(V_\nu^{n+1},V_\nu^{n+1})
    &=
    -
    \frac{r_{\mathrm{g}}^{n+1}}{S_{\mathrm{g}}^n}
    (w_{\mathrm{g}}^n,V_\nu^{n+1})^n
    -
    \sum_{i=1}^K
    \lambda_i^{n+1}
    (g_i^n,V_\nu^{n+1})^n .
    \label{eq:proof_time_discrete_normal_test}
\end{align}
Multiplying \eqref{eq:time_discrete_rg} by
$2r_{\mathrm{g}}^{n+1}$ and using the identity above yields
\begin{align}
    \frac{
        (r_{\mathrm{g}}^{n+1})^2
        -
        (r_{\mathrm{g}}^n)^2
        +
        (r_{\mathrm{g}}^{n+1}-r_{\mathrm{g}}^n)^2
    }{\Delta t}
    &=
    \frac{r_{\mathrm{g}}^{n+1}}{S_{\mathrm{g}}^n}
    (w_{\mathrm{g}}^n,V_\nu^{n+1})^n
    +
    \sum_{i=1}^K
    \lambda_i^{n+1}
    (g_i^n,V_\nu^{n+1})^n .
    \label{eq:proof_time_discrete_rg}
\end{align}
Adding \eqref{eq:proof_time_discrete_normal_test} and
\eqref{eq:proof_time_discrete_rg}, the scaled geometric force term and all
Lagrange multiplier terms cancel. Hence
\[
    \frac{
        (r_{\mathrm{g}}^{n+1})^2
        -
        (r_{\mathrm{g}}^n)^2
    }{\Delta t}
    +
    \frac{
        (r_{\mathrm{g}}^{n+1}-r_{\mathrm{g}}^n)^2
    }{\Delta t}
    +
    \mathcal{G}_\nu^n(V_\nu^{n+1},V_\nu^{n+1})
    +
    \Delta t\,
    \mathcal{D}_\nu^n(V_\nu^{n+1},V_\nu^{n+1})
    =
    0 .
\]
By the positivity properties of the metric and stabilization forms, the last three terms are
nonnegative. Dropping the nonnegative temporal variation term gives
\eqref{eq:time_discrete_geometric_sav_dissipation}.

The tangential mesh subsystem is treated in the same way. Taking
$\varphi_\tau=V_\tau^{n+1}$ in \eqref{eq:time_discrete_tangent} gives
\[
    \mathcal{G}_\tau^n(V_\tau^{n+1},V_\tau^{n+1})
    +
    \Delta t\,
    \mathcal{D}_\tau^n(V_\tau^{n+1},V_\tau^{n+1})
    =
    -
    \frac{r_{\mathrm{m}}^{n+1}}{S_{\mathrm{m}}^n}
    (w_{\mathrm{m}}^n,V_\tau^{n+1})^n .
\]
Multiplying \eqref{eq:time_discrete_rm} by
$2r_{\mathrm{m}}^{n+1}$ yields
\[
    \frac{
        (r_{\mathrm{m}}^{n+1})^2
        -
        (r_{\mathrm{m}}^n)^2
        +
        (r_{\mathrm{m}}^{n+1}-r_{\mathrm{m}}^n)^2
    }{\Delta t}
    =
    \frac{r_{\mathrm{m}}^{n+1}}{S_{\mathrm{m}}^n}
    (w_{\mathrm{m}}^n,V_\tau^{n+1})^n .
\]
Adding these two identities, we obtain
\[
    \frac{
        (r_{\mathrm{m}}^{n+1})^2
        -
        (r_{\mathrm{m}}^n)^2
    }{\Delta t}
    +
    \frac{
        (r_{\mathrm{m}}^{n+1}-r_{\mathrm{m}}^n)^2
    }{\Delta t}
    +
    \mathcal{G}_\tau^n(V_\tau^{n+1},V_\tau^{n+1})
    +
    \Delta t\,
    \mathcal{D}_\tau^n(V_\tau^{n+1},V_\tau^{n+1})
    =
    0 .
\]
Using the positivity properties of the metric and stabilization forms, and dropping the nonnegative
temporal variation term, we obtain
\eqref{eq:time_discrete_mesh_sav_dissipation}. This completes the proof.
\end{proof}

\begin{remark}[What is dissipated]
Theorem~\ref{thm:time_discrete_modified_energy_dissipation} gives a monotonicity
result for the modified SAV quantities
$(r_{\mathrm{g}}^n)^2$ and $(r_{\mathrm{m}}^n)^2$. These quantities should not
be identified automatically with the original energies
$E_{\mathrm{geom}}(\bm{\gamma}^n)+C_{\mathrm{g}}$ and
$E_{\mathrm{mesh}}(\bm{\gamma}^n)+C_{\mathrm{m}}$ at the time-discrete level.
The relation between the modified SAV energies and the original geometric
energies is therefore monitored in the numerical experiments when relevant.
\end{remark}

\begin{remark}[Role of nonlinear constraints]
The nonlinear constraints \eqref{eq:time_discrete_constraints} are part of the
time-discrete system. However, they are not used directly in the algebraic
proof of Theorem~\ref{thm:time_discrete_modified_energy_dissipation}. The
Lagrange multiplier terms cancel because of the coupling between the SAV
equation \eqref{eq:time_discrete_rg} and the normal velocity equation
\eqref{eq:time_discrete_normal}. The constraints determine the Lagrange
multipliers through the nonlinear algebraic system obtained after spatial
discretization.
\end{remark}

\section{Spatial discretization and algebraic decoupling}
\label{sec:spatial_discretization}

We now pass from the time-discrete variational formulation to a fully discrete
parametric finite element scheme. The purpose of this section is twofold. First,
we show that the algebraic cancellation structure of the dual-SAV formulation is
preserved after spatial discretization. Second, we derive an algebraic block
reduction that separates the large linear spatial response from the
low-dimensional nonlinear constraint enforcement.

The resulting fully discrete system contains the spatial velocity coefficients,
the scalar auxiliary variables, and the Lagrange multipliers. A direct
monolithic solution would lead to a nonlinear algebraic system involving all
spatial degrees of freedom. The block reduction derived below does not modify
the fully discrete equations. Rather, it exploits the linear dependence of the
velocity equations on the scalar unknowns and reduces the nonlinear part to a
system for the geometric SAV variable and the Lagrange multipliers.

\subsection{Parametric finite element approximation}

Let $\bm{\gamma}_h^n$ be a polygonal approximation of the closed curve at time
level $n$. We write its vertices as
\[
    \{\mathbf{x}_j^n\}_{j=1}^N\subset\mathbb{R}^2,
    \qquad
    \mathbf{x}_{N+1}^n=\mathbf{x}_1^n,
\]
and connect consecutive vertices by straight line segments. The corresponding
periodic piecewise linear finite element space is denoted by
$\mathbb{V}_h^n\subset\mathbb{V}^n$. It is spanned by the nodal basis functions
$\{\chi_j^n\}_{j=1}^N$, which satisfy
\[
    \chi_j^n(\mathbf{x}_k^n)=\delta_{jk},
    \qquad
    j,k=1,\ldots,N .
\]

We approximate the scalar normal and tangential velocities by
\[
    V_{\nu,h}^{n+1}
    =
    \sum_{j=1}^N
    [\mathbf{V}_\nu^{n+1}]_j\,\chi_j^n,
    \qquad
    V_{\tau,h}^{n+1}
    =
    \sum_{j=1}^N
    [\mathbf{V}_\tau^{n+1}]_j\,\chi_j^n,
\]
where
$\mathbf{V}_\nu^{n+1},\mathbf{V}_\tau^{n+1}\in \mathbb{R}^N$
are the coefficient vectors.

For each edge, we define the edge vector, length, and unit tangent by
\[
    \mathbf{e}_j^n
    :=
    \mathbf{x}_{j+1}^n-\mathbf{x}_j^n,
    \qquad
    l_j^n
    :=
    |\mathbf{e}_j^n|,
    \qquad
    \bm{\tau}_{j+1/2}^n
    :=
    \frac{\mathbf{e}_j^n}{l_j^n}.
\]
The nodal tangent is then defined by the normalized average of the two adjacent
edge tangents,
\[
    \bm{\tau}_j^n
    :=
    \frac{
        \bm{\tau}_{j-1/2}^n+\bm{\tau}_{j+1/2}^n
    }{
        \left|
        \bm{\tau}_{j-1/2}^n+\bm{\tau}_{j+1/2}^n
        \right|
    },
    \qquad j=1,\ldots,N,
\]
whenever the denominator is nonzero. We assume that the polygonal curve is
regular in the sense that this denominator does not vanish. The nodal normal is
defined by applying the same $\pi/2$ rotation convention as in the continuous
setting:
\[
    \bm{\nu}_j^n
    :=
    \mathbf{R}_{\pi/2}\bm{\tau}_j^n,
\]
where $\mathbf{R}_{\pi/2}$ is chosen so that
$(\bm{\tau}_j^n,\bm{\nu}_j^n)$ is a positively oriented orthonormal basis.
Other consistent choices of nodal tangents and normals are possible, but the
fully discrete algebraic structure derived below only requires that these
vectors are fixed during the time step and computed from the known curve
$\bm{\gamma}_h^n$.

The updated polygonal curve is defined by
\[
    \bm{\gamma}_h^{n+1}
    =
    \bm{\gamma}_h^n
    +
    \Delta t
    \left(
        V_{\nu,h}^{n+1}\bm{\nu}_h^n
        +
        V_{\tau,h}^{n+1}\bm{\tau}_h^n
    \right).
\]
Equivalently, at the nodal level,
\[
    \mathbf{x}_j^{n+1}
    =
    \mathbf{x}_j^n
    +
    \Delta t
    \left(
        [\mathbf{V}_\nu^{n+1}]_j\bm{\nu}_j^n
        +
        [\mathbf{V}_\tau^{n+1}]_j\bm{\tau}_j^n
    \right),
    \qquad
    j=1,\ldots,N.
\]

The frozen metric and stabilization forms are restricted to the finite element
space $\mathbb{V}_h^n$. For all test functions
$\chi_k^n\in\mathbb{V}_h^n$, the fully discrete velocity equations
corresponding to \eqref{eq:time_discrete_normal} and
\eqref{eq:time_discrete_tangent} read
\begin{align}
    \mathcal{G}_{\nu,h}^n
    (V_{\nu,h}^{n+1},\chi_k^n)
    +
    \Delta t\,
    \mathcal{D}_{\nu,h}^n
    (V_{\nu,h}^{n+1},\chi_k^n)
    &=
    -
    \frac{r_{\mathrm{g},h}^{n+1}}{S_{\mathrm{g},h}^n}
    (w_{\mathrm{g},h}^n,\chi_k^n)_h^n
    -
    \sum_{i=1}^K
    \lambda_{i,h}^{n+1}
    (g_{i,h}^n,\chi_k^n)_h^n,
    \label{eq:fully_discrete_normal_variational}
    \\
    \mathcal{G}_{\tau,h}^n
    (V_{\tau,h}^{n+1},\chi_k^n)
    +
    \Delta t\,
    \mathcal{D}_{\tau,h}^n
    (V_{\tau,h}^{n+1},\chi_k^n)
    &=
    -
    \frac{r_{\mathrm{m},h}^{n+1}}{S_{\mathrm{m},h}^n}
    (w_{\mathrm{m},h}^n,\chi_k^n)_h^n .
    \label{eq:fully_discrete_tangent_variational}
\end{align}
Here
$S_{\mathrm{g},h}^n:=S_{\mathrm{g}}(\bm{\gamma}_h^n)$,
$S_{\mathrm{m},h}^n:=S_{\mathrm{m}}(\bm{\gamma}_h^n)$,
and $(\cdot,\cdot)_h^n$ denotes the discrete $L^2$ inner product on
$\bm{\gamma}_h^n$. The quantities
$w_{\mathrm{g},h}^n$, $w_{\mathrm{m},h}^n$, and $g_{i,h}^n$ are evaluated
explicitly on the frozen geometry $\bm{\gamma}_h^n$. Consistent with the
zero-order stabilized time discretization, no velocity-history term appears in
\eqref{eq:fully_discrete_normal_variational} or
\eqref{eq:fully_discrete_tangent_variational}.

\subsection{Fully discrete matrix formulation}
\label{subsec:fully_discrete_matrix_formulation}

We now rewrite the finite element velocity equations in matrix form. Define the
metric matrices
\(\mathbf{A}_{\nu,h}^n,\mathbf{A}_{\tau,h}^n\in\mathbb{R}^{N\times N}\)
by
\[
    [\mathbf{A}_{\nu,h}^n]_{kj}
    :=
    \mathcal{G}_{\nu,h}^n(\chi_j^n,\chi_k^n),
    \qquad
    [\mathbf{A}_{\tau,h}^n]_{kj}
    :=
    \mathcal{G}_{\tau,h}^n(\chi_j^n,\chi_k^n).
\]
Similarly, define the stabilization matrices
\(\mathbf{D}_{\nu,h}^n,\mathbf{D}_{\tau,h}^n\in\mathbb{R}^{N\times N}\)
by
\[
    [\mathbf{D}_{\nu,h}^n]_{kj}
    :=
    \mathcal{D}_{\nu,h}^n(\chi_j^n,\chi_k^n),
    \qquad
    [\mathbf{D}_{\tau,h}^n]_{kj}
    :=
    \mathcal{D}_{\tau,h}^n(\chi_j^n,\chi_k^n).
\]

We next specify the matrix realization of the nonlocal \(H^{-1}\) metric and
of the biharmonic-type stabilization used below. Let \(\mathbf{K}_h^n\) be the
stiffness matrix on the frozen polygonal curve \(\bm{\gamma}_h^n\), and let
\(\mathbf{M}_{\mathrm{L},h}^n\) be the lumped mass matrix. We denote by
\(\mathbf{m}^n\) the vector of diagonal entries of \(\mathbf{M}_{\mathrm{L},h}^n\). Since
\(\mathbf{K}_h^n\) is symmetric positive semi-definite and has constants in
its kernel, the \(H^{-1}\) normal metric is used on the lumped-mass discrete
zero-mean subspace
\[
    \mathbb{V}_{h,0}^n
    :=
    \left\{
        \mathbf{u}\in\mathbb{R}^N:
        (\mathbf{m}^n)^\top\mathbf{u}=0
    \right\}.
\]
On this subspace, the discrete \(H^{-1}\) metric matrix is defined by
\begin{equation}
    \mathbf{A}_{H^{-1},h}^n
    :=
    \mathbf{M}_{\mathrm{L},h}^n
    (\mathbf{K}_h^n)^\dagger
    \mathbf{M}_{\mathrm{L},h}^n ,
    \label{eq:discrete_Hminus1_metric_matrix}
\end{equation}
where \((\mathbf{K}_h^n)^\dagger\) denotes the Moore--Penrose inverse of
\(\mathbf{K}_h^n\). Equivalently, this matrix represents the inverse
Laplace--Beltrami operator on the discrete zero-mean subspace. When the normal
metric is the \(H^{-1}\) metric, \(\mathbf{A}_{\nu,h}^n\) is understood as
\(\mathbf{A}_{H^{-1},h}^n\) restricted to \(\mathbb{V}_{h,0}^n\). The normal
right-hand sides are projected onto the dual of this subspace before solving
the response problems, and the computed normal response vectors are taken in
\(\mathbb{V}_{h,0}^n\).

For fourth-order geometric flows, the continuous biharmonic-type stabilization
introduced in Section~\ref{subsec:examples_metric_stabilization} is realized
algebraically. Since the spatial discretization uses piecewise linear finite
elements, we do not differentiate basis functions twice. Instead, we use the
squared discrete Laplace--Beltrami operator. In particular, the biharmonic-type
normal stabilization is represented by
\[
    \mathbf{D}_{\mathrm{bi},h}^n
    :=
    \beta_{\nu,4}
    \mathbf{K}_h^n
    (\mathbf{M}_{\mathrm{L},h}^n)^{-1}
    \mathbf{K}_h^n ,
    \qquad
    \beta_{\nu,4}\geq 0.
\]
For the curve diffusion flow, which uses the nonlocal \(H^{-1}\) normal
metric, we use the hybrid normal stabilization
\begin{equation}
    \mathbf{D}_{\nu,h}^n
    =
    \beta_{\nu,4}
    \mathbf{K}_h^n
    (\mathbf{M}_{\mathrm{L},h}^n)^{-1}
    \mathbf{K}_h^n
    +
    \beta_{\nu,2}
    \mathbf{K}_h^n,
    \qquad
    \beta_{\nu,4},\beta_{\nu,2}\geq 0 .
    \label{eq:hybrid_normal_stabilization}
\end{equation}
The first term is the biharmonic-type stabilization, and the second term is an
additional Laplacian stabilization. Both terms are symmetric positive
semi-definite. Indeed, for any \(\mathbf{u}\in\mathbb{R}^N\),
\[
    \mathbf{u}^{\top}
    \mathbf{K}_h^n
    (\mathbf{M}_{\mathrm{L},h}^n)^{-1}
    \mathbf{K}_h^n
    \mathbf{u}
    =
    (\mathbf{K}_h^n\mathbf{u})^{\top}
    (\mathbf{M}_{\mathrm{L},h}^n)^{-1}
    (\mathbf{K}_h^n\mathbf{u})
    \geq 0,
\]
and
\[
    \mathbf{u}^{\top}\mathbf{K}_h^n\mathbf{u}\geq 0.
\]
Thus, the hybrid choice satisfies the abstract assumption that
\(\mathbf{D}_{\nu,h}^n\) is symmetric positive semi-definite. Consequently, it
does not change the algebraic SAV cancellation or the block reduction. Its role
is purely numerical: the biharmonic component damps high-frequency modes
associated with fourth-order stiffness, while the Laplacian component damps
intermediate-frequency oscillations in the \(H^{-1}\) response.

The zero-order stabilized system matrices are
\[
    \mathbf{M}_{\nu,h}^n
    :=
    \mathbf{A}_{\nu,h}^n
    +
    \Delta t\,\mathbf{D}_{\nu,h}^n,
    \qquad
    \mathbf{M}_{\tau,h}^n
    :=
    \mathbf{A}_{\tau,h}^n
    +
    \Delta t\,\mathbf{D}_{\tau,h}^n .
\]
Under the assumptions on the metric and stabilization forms stated in
Section~\ref{subsec:abstract_components}, these matrices are
symmetric positive definite on the relevant discrete velocity spaces. When the
normal metric is the \(H^{-1}\) metric, this statement is understood after
restricting \(\mathbf{M}_{\nu,h}^n\) to the discrete zero-mean subspace
\(\mathbb{V}_{h,0}^n\).

We also define the load vectors associated with the geometric driving force,
the mesh driving force, and the constraint gradients:
\[
\begin{aligned}
    [\mathbf{F}_{\mathrm{g},h}^n]_k
    &:=
    (w_{\mathrm{g},h}^n,\chi_k^n)_h^n,
    &
    [\mathbf{F}_{\mathrm{m},h}^n]_k
    &:=
    (w_{\mathrm{m},h}^n,\chi_k^n)_h^n,
    \\
    [\mathbf{G}_{i,h}^n]_k
    &:=
    (g_{i,h}^n,\chi_k^n)_h^n,
    &
    i&=1,\ldots,K .
\end{aligned}
\]
Here the vectors
\(\mathbf{F}_{\mathrm{g},h}^n\),
\(\mathbf{F}_{\mathrm{m},h}^n\), and
\(\mathbf{G}_{i,h}^n\) are evaluated on the frozen polygonal curve
\(\bm{\gamma}_h^n\).

Let
\(\mathbf{V}_{\nu}^{n+1},\mathbf{V}_{\tau}^{n+1}\in\mathbb{R}^N\)
be the coefficient vectors of \(V_{\nu,h}^{n+1}\) and
\(V_{\tau,h}^{n+1}\) in the nodal basis. Then
\eqref{eq:fully_discrete_normal_variational} and
\eqref{eq:fully_discrete_tangent_variational} become
\begin{equation}
    \mathbf{M}_{\nu,h}^n
    \mathbf{V}_{\nu}^{n+1}
    =
    -
    \frac{r_{\mathrm{g},h}^{n+1}}{S_{\mathrm{g},h}^n}
    \mathbf{F}_{\mathrm{g},h}^n
    -
    \sum_{i=1}^K
    \lambda_{i,h}^{n+1}
    \mathbf{G}_{i,h}^n ,
    \label{eq:matrix_normal_velocity}
\end{equation}
and
\begin{equation}
    \mathbf{M}_{\tau,h}^n
    \mathbf{V}_{\tau}^{n+1}
    =
    -
    \frac{r_{\mathrm{m},h}^{n+1}}{S_{\mathrm{m},h}^n}
    \mathbf{F}_{\mathrm{m},h}^n .
    \label{eq:matrix_tangent_velocity}
\end{equation}
The normal velocity equation depends linearly on
\(r_{\mathrm{g},h}^{n+1}\) and on the Lagrange multipliers
$\lambda_{1,h}^{n+1},\ldots,\lambda_{K,h}^{n+1}$.
The tangential velocity equation depends linearly on
\(r_{\mathrm{m},h}^{n+1}\). These linear dependences will be used in the
algebraic block reduction below.

\subsection{Fully discrete algebraic system}
\label{subsec:fully_discrete_algebraic_system}

Using the matrices and vectors defined above, we now write the fully discrete
system at one time step. Given
\(\bm{\gamma}_h^n\), \(r_{\mathrm{g},h}^n\), and \(r_{\mathrm{m},h}^n\), we seek
\[
    \mathbf{V}_\nu^{n+1},\mathbf{V}_\tau^{n+1}\in\mathbb{R}^N,
    \qquad
    r_{\mathrm{g},h}^{n+1},r_{\mathrm{m},h}^{n+1}\in\mathbb{R},
    \qquad
    \lambda_{1,h}^{n+1},\ldots,\lambda_{K,h}^{n+1}\in\mathbb{R},
\]
such that
\begin{subequations}
\label{eq:fully_discrete_system}
\begin{align}
    \mathbf{M}_{\nu,h}^n\mathbf{V}_{\nu}^{n+1}
    &=
    -
    \frac{r_{\mathrm{g},h}^{n+1}}{S_{\mathrm{g},h}^n}
    \mathbf{F}_{\mathrm{g},h}^n
    -
    \sum_{i=1}^K
    \lambda_{i,h}^{n+1}
    \mathbf{G}_{i,h}^n,
    \label{eq:fully_discrete_system_normal}
    \\
    \mathbf{M}_{\tau,h}^n\mathbf{V}_{\tau}^{n+1}
    &=
    -
    \frac{r_{\mathrm{m},h}^{n+1}}{S_{\mathrm{m},h}^n}
    \mathbf{F}_{\mathrm{m},h}^n,
    \label{eq:fully_discrete_system_tangent}
    \\
    \frac{r_{\mathrm{g},h}^{n+1}-r_{\mathrm{g},h}^{n}}{\Delta t}
    &=
    \frac{1}{2S_{\mathrm{g},h}^n}
    (\mathbf{F}_{\mathrm{g},h}^n)^\top
    \mathbf{V}_{\nu}^{n+1}
    +
    \sum_{i=1}^K
    \frac{\lambda_{i,h}^{n+1}}{2r_{\mathrm{g},h}^{n+1}}
    (\mathbf{G}_{i,h}^n)^\top
    \mathbf{V}_{\nu}^{n+1},
    \label{eq:fully_discrete_system_rg}
    \\
    \frac{r_{\mathrm{m},h}^{n+1}-r_{\mathrm{m},h}^{n}}{\Delta t}
    &=
    \frac{1}{2S_{\mathrm{m},h}^n}
    (\mathbf{F}_{\mathrm{m},h}^n)^\top
    \mathbf{V}_{\tau}^{n+1}.
    \label{eq:fully_discrete_system_rm}
\end{align}
The updated nodal positions are defined by
\begin{equation}
    \mathbf{x}_j^{n+1}
    =
    \mathbf{x}_j^n
    +
    \Delta t
    \left(
        [\mathbf{V}_{\nu}^{n+1}]_j\bm{\nu}_j^n
        +
        [\mathbf{V}_{\tau}^{n+1}]_j\bm{\tau}_j^n
    \right),
    \qquad
    j=1,\ldots,N.
    \label{eq:fully_discrete_system_update}
\end{equation}
The nonlinear global constraints are imposed on the updated polygonal curve:
\begin{equation}
    C_i(\bm{\gamma}_h^{n+1})=0,
    \qquad
    i=1,\ldots,K.
    \label{eq:fully_discrete_system_constraints}
\end{equation}
\end{subequations}

Equations
\eqref{eq:fully_discrete_system}
constitute the fully discrete nonlinear algebraic system. The first two
equations are linear in the velocity vectors once the scalar variables are
fixed. The nonlinearity enters through the SAV update equation
\eqref{eq:fully_discrete_system_rg}, through the factor
$1/r_{\mathrm{g},h}^{n+1}$, and through the nonlinear constraints imposed on
the updated curve. This structure will be exploited in the block reduction
below.

\subsection{Fully discrete SAV dissipation}

We next prove that the fully discrete algebraic system preserves the SAV
dissipation structure. The proof is purely algebraic and relies on the
consistent use of the frozen matrices and frozen load vectors in the velocity
equations and in the SAV update equations.

\begin{theorem}[Fully discrete modified-energy dissipation]
\label{thm:fully_discrete_modified_energy_dissipation}
Under the assumptions on the metric and stabilization forms stated in
Sections~\ref{subsec:abstract_components} and
\ref{subsec:examples_metric_stabilization}, suppose that
$r_{\mathrm{g},h}^{n+1}\neq 0$ and that
$
    \mathbf{V}_{\nu}^{n+1}
$,
$
    \mathbf{V}_{\tau}^{n+1}
$,
$
    r_{\mathrm{g},h}^{n+1}
$,
$
    r_{\mathrm{m},h}^{n+1}
$,
$
    \lambda_{1,h}^{n+1},\ldots,\lambda_{K,h}^{n+1}
$
satisfy the fully discrete algebraic system
\eqref{eq:fully_discrete_system}.
Then the following estimates hold:
\begin{align}
    \frac{
        (r_{\mathrm{g},h}^{n+1})^2
        -
        (r_{\mathrm{g},h}^{n})^2
    }{\Delta t}
    &\leq
    -
    (\mathbf{V}_{\nu}^{n+1})^\top
    \mathbf{A}_{\nu,h}^n
    \mathbf{V}_{\nu}^{n+1}
    -
    \Delta t\,
    (\mathbf{V}_{\nu}^{n+1})^\top
    \mathbf{D}_{\nu,h}^n
    \mathbf{V}_{\nu}^{n+1}
    \leq 0,
    \label{eq:fully_discrete_geometric_sav_dissipation}
    \\
    \frac{
        (r_{\mathrm{m},h}^{n+1})^2
        -
        (r_{\mathrm{m},h}^{n})^2
    }{\Delta t}
    &\leq
    -
    (\mathbf{V}_{\tau}^{n+1})^\top
    \mathbf{A}_{\tau,h}^n
    \mathbf{V}_{\tau}^{n+1}
    -
    \Delta t\,
    (\mathbf{V}_{\tau}^{n+1})^\top
    \mathbf{D}_{\tau,h}^n
    \mathbf{V}_{\tau}^{n+1}
    \leq 0.
    \label{eq:fully_discrete_mesh_sav_dissipation}
\end{align}
Consequently, the fully discrete modified geometric and mesh SAV energies are
nonincreasing.
\end{theorem}

\begin{proof}
We use the elementary identity
\[
    2a(a-b)=a^2-b^2+(a-b)^2 .
\]

We first prove the estimate for the geometric SAV variable. Taking the
Euclidean inner product of the normal velocity equation
\eqref{eq:fully_discrete_system_normal} with
$\mathbf{V}_{\nu}^{n+1}$ gives
\begin{align}
    (\mathbf{V}_{\nu}^{n+1})^\top
    \mathbf{M}_{\nu,h}^n
    \mathbf{V}_{\nu}^{n+1}
    &=
    -
    \frac{r_{\mathrm{g},h}^{n+1}}{S_{\mathrm{g},h}^n}
    (\mathbf{F}_{\mathrm{g},h}^n)^\top
    \mathbf{V}_{\nu}^{n+1}
    -
    \sum_{i=1}^K
    \lambda_{i,h}^{n+1}
    (\mathbf{G}_{i,h}^n)^\top
    \mathbf{V}_{\nu}^{n+1}.
    \label{eq:proof_fully_discrete_normal_test}
\end{align}
On the other hand, multiplying the geometric SAV equation
\eqref{eq:fully_discrete_system_rg} by
$2r_{\mathrm{g},h}^{n+1}$ and using the identity above yields
\begin{align}
    \frac{
        (r_{\mathrm{g},h}^{n+1})^2
        -
        (r_{\mathrm{g},h}^{n})^2
        +
        (r_{\mathrm{g},h}^{n+1}-r_{\mathrm{g},h}^{n})^2
    }{\Delta t}
    &=
    \frac{r_{\mathrm{g},h}^{n+1}}{S_{\mathrm{g},h}^n}
    (\mathbf{F}_{\mathrm{g},h}^n)^\top
    \mathbf{V}_{\nu}^{n+1}
    +
    \sum_{i=1}^K
    \lambda_{i,h}^{n+1}
    (\mathbf{G}_{i,h}^n)^\top
    \mathbf{V}_{\nu}^{n+1}.
    \label{eq:proof_fully_discrete_rg}
\end{align}
Adding \eqref{eq:proof_fully_discrete_normal_test} and
\eqref{eq:proof_fully_discrete_rg}, the scaled geometric force term and all
Lagrange multiplier terms cancel. Hence
\[
    \frac{
        (r_{\mathrm{g},h}^{n+1})^2
        -
        (r_{\mathrm{g},h}^{n})^2
    }{\Delta t}
    +
    \frac{
        (r_{\mathrm{g},h}^{n+1}-r_{\mathrm{g},h}^{n})^2
    }{\Delta t}
    +
    (\mathbf{V}_{\nu}^{n+1})^\top
    \mathbf{M}_{\nu,h}^n
    \mathbf{V}_{\nu}^{n+1}
    =
    0 .
\]
Using
\[
    \mathbf{M}_{\nu,h}^n
    =
    \mathbf{A}_{\nu,h}^n
    +
    \Delta t\,\mathbf{D}_{\nu,h}^n,
\]
we obtain
\begin{align*}
    \frac{
        (r_{\mathrm{g},h}^{n+1})^2
        -
        (r_{\mathrm{g},h}^{n})^2
    }{\Delta t}
    &+
    \frac{
        (r_{\mathrm{g},h}^{n+1}-r_{\mathrm{g},h}^{n})^2
    }{\Delta t}
    +
    (\mathbf{V}_{\nu}^{n+1})^\top
    \mathbf{A}_{\nu,h}^n
    \mathbf{V}_{\nu}^{n+1}
    +
    \Delta t\,
    (\mathbf{V}_{\nu}^{n+1})^\top
    \mathbf{D}_{\nu,h}^n
    \mathbf{V}_{\nu}^{n+1}
    =
    0 .
\end{align*}
By the positivity properties of the metric and stabilization matrices, the
second, third, and fourth terms are nonnegative. Therefore,
\eqref{eq:fully_discrete_geometric_sav_dissipation} follows.

We next prove the estimate for the mesh SAV variable. Taking the Euclidean
inner product of the tangential velocity equation
\eqref{eq:fully_discrete_system_tangent} with
$\mathbf{V}_{\tau}^{n+1}$ gives
\[
    (\mathbf{V}_{\tau}^{n+1})^\top
    \mathbf{M}_{\tau,h}^n
    \mathbf{V}_{\tau}^{n+1}
    =
    -
    \frac{r_{\mathrm{m},h}^{n+1}}{S_{\mathrm{m},h}^n}
    (\mathbf{F}_{\mathrm{m},h}^n)^\top
    \mathbf{V}_{\tau}^{n+1}.
\]
Multiplying the mesh SAV equation
\eqref{eq:fully_discrete_system_rm} by
$2r_{\mathrm{m},h}^{n+1}$ gives
\[
    \frac{
        (r_{\mathrm{m},h}^{n+1})^2
        -
        (r_{\mathrm{m},h}^{n})^2
        +
        (r_{\mathrm{m},h}^{n+1}-r_{\mathrm{m},h}^{n})^2
    }{\Delta t}
    =
    \frac{r_{\mathrm{m},h}^{n+1}}{S_{\mathrm{m},h}^n}
    (\mathbf{F}_{\mathrm{m},h}^n)^\top
    \mathbf{V}_{\tau}^{n+1}.
\]
Adding these two identities yields
\[
    \frac{
        (r_{\mathrm{m},h}^{n+1})^2
        -
        (r_{\mathrm{m},h}^{n})^2
    }{\Delta t}
    +
    \frac{
        (r_{\mathrm{m},h}^{n+1}-r_{\mathrm{m},h}^{n})^2
    }{\Delta t}
    +
    (\mathbf{V}_{\tau}^{n+1})^\top
    \mathbf{M}_{\tau,h}^n
    \mathbf{V}_{\tau}^{n+1}
    =
    0 .
\]
Since
\[
    \mathbf{M}_{\tau,h}^n
    =
    \mathbf{A}_{\tau,h}^n
    +
    \Delta t\,\mathbf{D}_{\tau,h}^n,
\]
we have
\begin{align*}
    \frac{
        (r_{\mathrm{m},h}^{n+1})^2
        -
        (r_{\mathrm{m},h}^{n})^2
    }{\Delta t}
    &+
    \frac{
        (r_{\mathrm{m},h}^{n+1}-r_{\mathrm{m},h}^{n})^2
    }{\Delta t}
    +
    (\mathbf{V}_{\tau}^{n+1})^\top
    \mathbf{A}_{\tau,h}^n
    \mathbf{V}_{\tau}^{n+1}
    +
    \Delta t\,
    (\mathbf{V}_{\tau}^{n+1})^\top
    \mathbf{D}_{\tau,h}^n
    \mathbf{V}_{\tau}^{n+1}
    =
    0 .
\end{align*}
Again, by the positivity properties of the metric and stabilization matrices,
the second, third, and fourth terms are nonnegative. Hence
\eqref{eq:fully_discrete_mesh_sav_dissipation} follows. The proof is complete.
\end{proof}

\begin{remark}[Role of nonlinear constraints]
The nonlinear constraints
\eqref{eq:fully_discrete_system_constraints}
are part of the fully discrete algebraic system. However, they are not used
directly in the proof of
Theorem~\ref{thm:fully_discrete_modified_energy_dissipation}. The dissipation
estimate follows from the cancellation between the velocity equations and the
SAV update equations. The constraints determine the Lagrange multipliers
through the reduced nonlinear system derived below.
\end{remark}

\subsection{Algebraic block reduction}
\label{subsec:block_reduction}

We now exploit the linear dependence of the velocity equations
\eqref{eq:fully_discrete_system_normal}--\eqref{eq:fully_discrete_system_tangent}
on the scalar unknowns. The block reduction is an algebraic reformulation of
the fully discrete system; it does not introduce an additional approximation.

We first define the response vectors associated with the frozen normal
equation. For the $L^2$ normal metric, these vectors are obtained by solving
\[
    \mathbf{M}_{\nu,h}^n\mathbf{V}_{\mathrm{g},h}^n
    =
    -\mathbf{F}_{\mathrm{g},h}^n,
    \qquad
    \mathbf{M}_{\nu,h}^n\mathbf{V}_{i,h}^n
    =
    -\mathbf{G}_{i,h}^n,
    \quad i=1,\ldots,K.
\]
For the $H^{-1}$ normal metric, the same notation is used, but the response
problems are understood on the discrete zero-mean subspace
$\mathbb{V}_{h,0}^n$. More precisely, let
\[
    \Pi_{0,h}^{n,*}\mathbf{F}
    :=
    \mathbf{F}
    -
    \frac{\mathbf{1}^{\top}\mathbf{F}}{\mathbf{1}^{\top}\mathbf{m}^n}
    \mathbf{m}^n
\]
denote the dual zero-mean projection, where $\mathbf{m}^n$ is the vector of diagonal
entries of the lumped mass matrix $\mathbf{M}_{\mathrm{L},h}^n$. Then, in the
$H^{-1}$ case, the normal response vectors are defined by
$
    \mathbf{V}_{\mathrm{g},h}^n\in\mathbb{V}_{h,0}^n
$,
$
    \mathbf{V}_{i,h}^n\in\mathbb{V}_{h,0}^n
$,
and
\[
    \mathbf{M}_{\nu,h}^n\mathbf{V}_{\mathrm{g},h}^n
    =
    -
    \Pi_{0,h}^{n,*}\mathbf{F}_{\mathrm{g},h}^n,
    \qquad
    \mathbf{M}_{\nu,h}^n\mathbf{V}_{i,h}^n
    =
    -
    \Pi_{0,h}^{n,*}\mathbf{G}_{i,h}^n,
    \quad i=1,\ldots,K,
\]
in the sense of equations posed on $\mathbb{V}_{h,0}^n$. Equivalently, these
systems may be solved with the side condition
$
    (\mathbf{m}^n)^\top\mathbf{V}=0
$.
With this convention, the notation below is common to both the $L^2$ and
$H^{-1}$ normal metrics.

The tangential response vector is defined by
\[
    \mathbf{M}_{\tau,h}^n\mathbf{V}_{\mathrm{m},h}^n
    =
    -\mathbf{F}_{\mathrm{m},h}^n .
\]

Once these response vectors are available, the normal and tangential velocity
equations imply the synthesis formulas
\begin{align}
    \mathbf{V}_{\nu}^{n+1}
    &=
    \frac{r_{\mathrm{g},h}^{n+1}}{S_{\mathrm{g},h}^n}
    \mathbf{V}_{\mathrm{g},h}^n
    +
    \sum_{i=1}^K
    \lambda_{i,h}^{n+1}
    \mathbf{V}_{i,h}^n,
    \label{eq:block_reduction_normal_synthesis}
    \\
    \mathbf{V}_{\tau}^{n+1}
    &=
    \frac{r_{\mathrm{m},h}^{n+1}}{S_{\mathrm{m},h}^n}
    \mathbf{V}_{\mathrm{m},h}^n .
    \label{eq:block_reduction_tangent_synthesis}
\end{align}
In particular, when the normal metric is the $H^{-1}$ metric, the synthesized
normal velocity belongs to $\mathbb{V}_{h,0}^n$.

The tangential subsystem is independent of the Lagrange multipliers. Substituting
\eqref{eq:block_reduction_tangent_synthesis} into
\eqref{eq:fully_discrete_system_rm} gives
\[
    \frac{r_{\mathrm{m},h}^{n+1}-r_{\mathrm{m},h}^{n}}{\Delta t}
    =
    \frac{r_{\mathrm{m},h}^{n+1}}{2(S_{\mathrm{m},h}^n)^2}
    (\mathbf{F}_{\mathrm{m},h}^n)^\top
    \mathbf{V}_{\mathrm{m},h}^n .
\]
Hence
\begin{equation}
    r_{\mathrm{m},h}^{n+1}
    =
    \frac{
        r_{\mathrm{m},h}^{n}
    }{
        1
        -
        \dfrac{\Delta t}{2(S_{\mathrm{m},h}^n)^2}
        (\mathbf{F}_{\mathrm{m},h}^n)^\top
        \mathbf{V}_{\mathrm{m},h}^n
    } .
    \label{eq:block_reduction_rm_update}
\end{equation}
Since
\[
    (\mathbf{F}_{\mathrm{m},h}^n)^\top
    \mathbf{V}_{\mathrm{m},h}^n
    =
    -
    (\mathbf{F}_{\mathrm{m},h}^n)^\top
    (\mathbf{M}_{\tau,h}^n)^{-1}
    \mathbf{F}_{\mathrm{m},h}^n
    \leq 0,
\]
the denominator in \eqref{eq:block_reduction_rm_update} is not smaller than
one. Thus, the mesh auxiliary variable is updated explicitly after the
tangential response solve. The tangential velocity is then reconstructed by
\[
    \mathbf{V}_{\tau}^{n+1}
    =
    \frac{r_{\mathrm{m},h}^{n+1}}{S_{\mathrm{m},h}^n}
    \mathbf{V}_{\mathrm{m},h}^n .
\]

It remains to determine
$
    r_{\mathrm{g},h}^{n+1}
$,
$
    \lambda_{1,h}^{n+1},\ldots,\lambda_{K,h}^{n+1}
$.
We collect these scalar unknowns into
\[
    \bm{\Xi}
    :=
    (r,\lambda_1,\ldots,\lambda_K)^\top
    \in\mathbb{R}^{K+1},
\]
where
\[
    r=r_{\mathrm{g},h}^{n+1},
    \qquad
    \lambda_i=\lambda_{i,h}^{n+1}.
\]
For a given $\bm{\Xi}$, define
\begin{equation}
    \mathbf{V}_{\nu}(\bm{\Xi})
    :=
    \frac{r}{S_{\mathrm{g},h}^n}
    \mathbf{V}_{\mathrm{g},h}^n
    +
    \sum_{i=1}^K
    \lambda_i\mathbf{V}_{i,h}^n .
    \label{eq:block_reduction_Vnu_Xi}
\end{equation}
The intermediate updated nodes are then given by
\begin{equation}
    \mathbf{x}_j(\bm{\Xi})
    =
    \mathbf{x}_j^n
    +
    \Delta t
    \left(
        [\mathbf{V}_{\nu}(\bm{\Xi})]_j\bm{\nu}_j^n
        +
        [\mathbf{V}_{\tau}^{n+1}]_j\bm{\tau}_j^n
    \right),
    \qquad
    j=1,\ldots,N.
    \label{eq:block_reduction_updated_nodes}
\end{equation}
The polygonal curve with vertices
$\{\mathbf{x}_j(\bm{\Xi})\}_{j=1}^N$ is denoted by
$\bm{\gamma}_h(\bm{\Xi})$.

The reduced nonlinear system consists of the geometric SAV equation and the
nonlinear constraints. Define
\begin{align}
    R_0(\bm{\Xi})
    &:=
    \frac{r-r_{\mathrm{g},h}^{n}}{\Delta t}
    -
    \frac{1}{2S_{\mathrm{g},h}^n}
    (\mathbf{F}_{\mathrm{g},h}^n)^\top
    \mathbf{V}_{\nu}(\bm{\Xi})
    -
    \sum_{i=1}^K
    \frac{\lambda_i}{2r}
    (\mathbf{G}_{i,h}^n)^\top
    \mathbf{V}_{\nu}(\bm{\Xi}),
    \label{eq:block_reduction_R0}
    \\
    R_i(\bm{\Xi})
    &:=
    C_i(\bm{\gamma}_h(\bm{\Xi})),
    \qquad i=1,\ldots,K.
    \label{eq:block_reduction_Ri}
\end{align}
Then the remaining scalar unknowns are determined by
\begin{equation}
    \bm{R}(\bm{\Xi})
    :=
    (R_0(\bm{\Xi}),R_1(\bm{\Xi}),\ldots,R_K(\bm{\Xi}))^\top
    =
    \bm{0}.
    \label{eq:block_reduction_reduced_system}
\end{equation}
This reduced nonlinear system has dimension $K+1$, independently of the number
of mesh vertices $N$.

\begin{remark}[Equivalence with the fully discrete system]
The block reduction is an algebraic reformulation of the fully discrete system,
not an additional approximation. Indeed, if $\bm\Xi$ solves
\eqref{eq:block_reduction_reduced_system}, and if
$r_{\mathrm{m},h}^{n+1}$, $\mathbf{V}_{\tau}^{n+1}$, and
$\mathbf{V}_{\nu}^{n+1}$ are reconstructed by
\eqref{eq:block_reduction_rm_update},
\eqref{eq:block_reduction_tangent_synthesis}, and
\eqref{eq:block_reduction_normal_synthesis}, then
\eqref{eq:fully_discrete_system_normal}--\eqref{eq:fully_discrete_system_constraints}
are satisfied. Conversely, any solution of the fully discrete system with
$r_{\mathrm{g},h}^{n+1}\neq 0$ satisfies the reduced system obtained above.
Thus, the reduced system is a solution strategy for the same fully discrete
equations.
\end{remark}

\begin{remark}[Compatibility of constraints with the $H^{-1}$ normal space]
For the $H^{-1}$ normal metric, the response vectors are computed in the
discrete zero-mean subspace. Consequently, only the components of the frozen
load and constraint-gradient vectors that act on this subspace contribute to
the normal response. In particular, constraints whose gradients lie entirely in
the excluded constant mode cannot be enforced by an ordinary Lagrange multiplier
within the strictly zero-mean $H^{-1}$ velocity space.

This is consistent with the continuous curve diffusion flow: area preservation
is encoded in the zero-mean character of the $H^{-1}$ normal velocity, rather
than imposed as an additional Lagrange multiplier constraint. Therefore, in the
curve diffusion experiment below, the area is monitored as a diagnostic
quantity unless an additional correction mechanism is explicitly introduced.
\end{remark}

\subsection{Newton solution of the reduced nonlinear system}
\label{subsec:newton_reduced_system}

We solve the reduced nonlinear system
\eqref{eq:block_reduction_reduced_system}
by Newton's method. Recall that the unknown vector is
\[
    \bm{\Xi}
    =
    (r,\lambda_1,\ldots,\lambda_K)^\top
    \in\mathbb{R}^{K+1},
\]
where
$
    r=r_{\mathrm{g},h}^{n+1}
$,
$
    \lambda_i=\lambda_{i,h}^{n+1}
$.
At the $m$-th Newton iteration, we solve
\begin{equation}
    \mathbf{J}(\bm{\Xi}^{(m)})\delta\bm{\Xi}^{(m)}
    =
    -\bm{R}(\bm{\Xi}^{(m)}),
    \qquad
    \bm{\Xi}^{(m+1)}
    =
    \bm{\Xi}^{(m)}+\delta\bm{\Xi}^{(m)}.
    \label{eq:newton_iteration_reduced_system}
\end{equation}
Here
\[
    \mathbf{J}(\bm{\Xi})
    :=
    \frac{\partial\bm{R}}{\partial\bm{\Xi}}(\bm{\Xi})
    \in\mathbb{R}^{(K+1)\times(K+1)}
\]
is the Jacobian matrix of the reduced residual.

The normal velocity synthesis
\eqref{eq:block_reduction_Vnu_Xi}
depends linearly on the scalar unknowns. Hence
\begin{equation}
    \frac{\partial\mathbf{V}_{\nu}}{\partial r}
    =
    \frac{1}{S_{\mathrm{g},h}^n}
    \mathbf{V}_{\mathrm{g},h}^n,
    \qquad
    \frac{\partial\mathbf{V}_{\nu}}{\partial\lambda_j}
    =
    \mathbf{V}_{j,h}^n,
    \quad j=1,\ldots,K.
    \label{eq:velocity_derivatives_scalar_unknowns}
\end{equation}
When the normal metric is the $H^{-1}$ metric, these response vectors have
already been computed in the discrete zero-mean subspace
$\mathbb{V}_{h,0}^n$. Therefore, the formulas below remain unchanged.

We first compute the derivatives of the SAV residual $R_0$. For compactness,
define
\[
    B_i(\bm{\Xi})
    :=
    (\mathbf{G}_{i,h}^n)^\top\mathbf{V}_{\nu}(\bm{\Xi}),
    \qquad
    i=1,\ldots,K.
\]
Differentiating \eqref{eq:block_reduction_R0} gives
\begin{align}
    \frac{\partial R_0}{\partial r}(\bm{\Xi})
    &=
    \frac{1}{\Delta t}
    -
    \frac{1}{2(S_{\mathrm{g},h}^n)^2}
    (\mathbf{F}_{\mathrm{g},h}^n)^\top
    \mathbf{V}_{\mathrm{g},h}^n
    +
    \sum_{i=1}^K
    \frac{\lambda_i}{2r^2}
    B_i(\bm{\Xi})
    -
    \sum_{i=1}^K
    \frac{\lambda_i}{2rS_{\mathrm{g},h}^n}
    (\mathbf{G}_{i,h}^n)^\top
    \mathbf{V}_{\mathrm{g},h}^n,
    \label{eq:jacobian_R0_r}
    \\
    \frac{\partial R_0}{\partial\lambda_j}(\bm{\Xi})
    &=
    -
    \frac{1}{2S_{\mathrm{g},h}^n}
    (\mathbf{F}_{\mathrm{g},h}^n)^\top
    \mathbf{V}_{j,h}^n
    -
    \frac{1}{2r}
    B_j(\bm{\Xi})
    -
    \sum_{i=1}^K
    \frac{\lambda_i}{2r}
    (\mathbf{G}_{i,h}^n)^\top
    \mathbf{V}_{j,h}^n,
    \qquad j=1,\ldots,K.
    \label{eq:jacobian_R0_lambda}
\end{align}
These expressions involve only the frozen load vectors and the response
vectors, except for the scalar quantities $B_i(\bm{\Xi})$.

We next compute the derivatives of the constraint residuals. Since
\[
    R_i(\bm{\Xi})
    =
    C_i(\bm{\gamma}_h(\bm{\Xi})),
    \qquad i=1,\ldots,K,
\]
the chain rule gives, for any scalar parameter
$
    \xi\in\{r,\lambda_1,\ldots,\lambda_K\}
$,
\[
    \frac{\partial R_i}{\partial\xi}(\bm{\Xi})
    =
    \sum_{j=1}^N
    \nabla_{\mathbf{x}_j}C_i(\bm{\gamma}_h(\bm{\Xi}))
    \cdot
    \frac{\partial\mathbf{x}_j(\bm{\Xi})}{\partial\xi}.
\]
By \eqref{eq:block_reduction_updated_nodes}, the tangential velocity
$\mathbf{V}_{\tau}^{n+1}$ has already been reconstructed and is independent of
$\bm{\Xi}$. Therefore,
\[
    \frac{\partial\mathbf{x}_j(\bm{\Xi})}{\partial\xi}
    =
    \Delta t\,
    \frac{\partial[\mathbf{V}_{\nu}(\bm{\Xi})]_j}{\partial\xi}
    \bm{\nu}_j^n .
\]
We define the projected constraint-gradient vector
$
    \widetilde{\mathbf{G}}_{i,h}(\bm{\Xi})\in\mathbb{R}^N
$
by
\begin{equation}
    [\widetilde{\mathbf{G}}_{i,h}(\bm{\Xi})]_j
    :=
    \nabla_{\mathbf{x}_j}C_i(\bm{\gamma}_h(\bm{\Xi}))
    \cdot
    \bm{\nu}_j^n,
    \qquad
    j=1,\ldots,N.
    \label{eq:projected_constraint_gradient}
\end{equation}
Then
\[
    \frac{\partial R_i}{\partial\xi}(\bm{\Xi})
    =
    \Delta t\,
    \widetilde{\mathbf{G}}_{i,h}(\bm{\Xi})^\top
    \frac{\partial\mathbf{V}_{\nu}}{\partial\xi}.
\]
Using \eqref{eq:velocity_derivatives_scalar_unknowns}, we obtain
\begin{align}
    \frac{\partial R_i}{\partial r}(\bm{\Xi})
    &=
    \frac{\Delta t}{S_{\mathrm{g},h}^n}
    \widetilde{\mathbf{G}}_{i,h}(\bm{\Xi})^\top
    \mathbf{V}_{\mathrm{g},h}^n,
    \label{eq:jacobian_constraint_r}
    \\
    \frac{\partial R_i}{\partial\lambda_j}(\bm{\Xi})
    &=
    \Delta t\,
    \widetilde{\mathbf{G}}_{i,h}(\bm{\Xi})^\top
    \mathbf{V}_{j,h}^n,
    \qquad i,j=1,\ldots,K.
    \label{eq:jacobian_constraint_lambda}
\end{align}

It is important to distinguish the frozen vectors
$\mathbf{G}_{i,h}^n$ from the projected gradients
$\widetilde{\mathbf{G}}_{i,h}(\bm{\Xi})$. The vectors
$\mathbf{G}_{i,h}^n$ are evaluated on the known curve $\bm{\gamma}_h^n$ and are
used in the normal velocity equation and in the geometric SAV equation. Their
role is to preserve the algebraic cancellation in the modified-energy
estimate. By contrast,
$\widetilde{\mathbf{G}}_{i,h}(\bm{\Xi})$ is evaluated on the intermediate updated
curve $\bm{\gamma}_h(\bm{\Xi})$ and appears only in the Jacobian of the
nonlinear constraint residuals.

In practice, all inner products involving only the frozen vectors
$\mathbf{F}_{\mathrm{g},h}^n$, $\mathbf{G}_{i,h}^n$,
$\mathbf{V}_{\mathrm{g},h}^n$, and $\mathbf{V}_{j,h}^n$ can be precomputed before the
Newton iteration. At each Newton step, one synthesizes
$\mathbf{V}_{\nu}(\bm{\Xi}^{(m)})$, constructs the intermediate curve
$\bm{\gamma}_h(\bm{\Xi}^{(m)})$, evaluates the constraint residuals
$R_i(\bm{\Xi}^{(m)})$, and assembles the projected gradients
$\widetilde{\mathbf{G}}_{i,h}(\bm{\Xi}^{(m)})$.

Unless otherwise stated, we use the initial guess
\[
    \bm{\Xi}^{(0)}
    =
    (r_{\mathrm{g},h}^{n},0,\ldots,0)^\top.
\]
When Lagrange multipliers from the previous time step are available, they may
be used as a warm start. The Newton iteration is stopped when
\[
    \|\bm{R}(\bm{\Xi}^{(m)})\|_2
    \leq
    \varepsilon_{\mathrm{Newton}}
\]
or when the relative increment satisfies
\[
    \frac{\|\delta\bm{\Xi}^{(m)}\|_2}
         {1+\|\bm{\Xi}^{(m)}\|_2}
    \leq
    \varepsilon_{\mathrm{Newton}}.
\]
Since the residual $R_0$ contains the factor $1/r$, the iteration is restricted
to iterates with $r\neq0$. In the computations below, the geometric auxiliary
variable is initialized from the positive quantity $S_{\mathrm{g},h}^0$, and
the Newton updates are monitored so that the iteration does not cross $r=0$.

\begin{remark}[Jacobian formulas for the $H^{-1}$ metric]
When the normal metric is the $H^{-1}$ metric, the response vectors
$\mathbf{V}_{\mathrm{g},h}^n$ and $\mathbf{V}_{i,h}^n$ are computed in
$\mathbb{V}_{h,0}^n$ with projected right-hand sides, as described in
Section~\ref{subsec:block_reduction}. After these response vectors have been
computed, the reduced residual and its Jacobian are assembled by the same
formulas as in the $L^2$ case. Thus, the low-dimensional Newton solve is
unchanged by the choice of the normal metric; the metric affects only the
precomputed response vectors.
\end{remark}

\begin{remark}[Local solvability of the reduced system]
The block reduction leaves a nonlinear system only for
\[
    \bm{\Xi}
    =
    (r,\lambda_1,\ldots,\lambda_K)^\top
    \in \mathbb{R}^{K+1}.
\]
The present paper does not attempt to prove a global solvability result for
this reduced system. Nevertheless, the formulation is locally well posed in the
standard algebraic sense. Namely, if the Jacobian of the reduced residual with
respect to \(\bm{\Xi}\) is nonsingular at a solution, then the implicit function
theorem gives local uniqueness and smooth dependence of the reduced unknowns
on the frozen geometric data and on the time step. In the numerical
experiments below, Newton's method is applied to this reduced system, and the
reported iteration counts provide a practical diagnostic of this local
solvability condition.
\end{remark}

\subsection{Algorithmic summary of one time step}
\label{subsec:algorithmic_summary}

We summarize one time step of the fully discrete dual-SAV PFEM in
Algorithm~\ref{alg:dual_sav_one_step}. The algorithm is written for the general
constrained case. For the $H^{-1}$ normal metric, the normal response solves are
understood on the discrete zero-mean subspace, as described above.

\begin{algorithm}[htbp]
\caption{One time step of the fully discrete dual-SAV PFEM}
\label{alg:dual_sav_one_step}
\begin{algorithmic}[1]
\Require
$\bm{\gamma}_h^n$, $r_{\mathrm{g},h}^n$, $r_{\mathrm{m},h}^n$,
and the time step size $\Delta t$
\Ensure
$\bm{\gamma}_h^{n+1}$, $r_{\mathrm{g},h}^{n+1}$,
$r_{\mathrm{m},h}^{n+1}$

\State
Compute frozen geometric quantities from $\bm{\gamma}_h^n$, including nodal
tangents and normals, metric matrices, stabilization matrices, and load vectors.

\State
Assemble the stabilized matrices
\[
    \mathbf{M}_{\nu,h}^n
    =
    \mathbf{A}_{\nu,h}^n+\Delta t\,\mathbf{D}_{\nu,h}^n,
    \qquad
    \mathbf{M}_{\tau,h}^n
    =
    \mathbf{A}_{\tau,h}^n+\Delta t\,\mathbf{D}_{\tau,h}^n .
\]
For fourth-order or nonlocal fourth-order flows, choose
$\mathbf{D}_{\nu,h}^n$ according to the appropriate symmetric positive
semi-definite stabilization, such as the hybrid stabilization
\eqref{eq:hybrid_normal_stabilization}.

\State
If the normal metric is the $H^{-1}$ metric, project the normal load vectors
onto the dual of the discrete zero-mean subspace.

\State
Solve the normal response problems
\[
    \mathbf{M}_{\nu,h}^n\mathbf{V}_{\mathrm{g},h}^n
    =
    -\mathbf{F}_{\mathrm{g},h}^n,
    \qquad
    \mathbf{M}_{\nu,h}^n\mathbf{V}_{i,h}^n
    =
    -\mathbf{G}_{i,h}^n,
    \quad i=1,\ldots,K,
\]
on the relevant normal velocity space. For the $H^{-1}$ metric, this means the
discrete zero-mean subspace.

\State
Solve the tangential response problem
\[
    \mathbf{M}_{\tau,h}^n\mathbf{V}_{\mathrm{m},h}^n
    =
    -\mathbf{F}_{\mathrm{m},h}^n .
\]

\State
Compute the mesh auxiliary variable by
\[
    r_{\mathrm{m},h}^{n+1}
    =
    \frac{
        r_{\mathrm{m},h}^{n}
    }{
        1
        -
        \dfrac{\Delta t}{2(S_{\mathrm{m},h}^n)^2}
        (\mathbf{F}_{\mathrm{m},h}^n)^\top
        \mathbf{V}_{\mathrm{m},h}^n
    } ,
\]
and reconstruct the tangential velocity
\[
    \mathbf{V}_{\tau}^{n+1}
    =
    \frac{r_{\mathrm{m},h}^{n+1}}{S_{\mathrm{m},h}^n}
    \mathbf{V}_{\mathrm{m},h}^n .
\]

\State
Solve the reduced nonlinear system
\[
    \bm{R}(\bm{\Xi})=\bm{0},
    \qquad
    \bm{\Xi}=(r,\lambda_1,\ldots,\lambda_K)^\top,
\]
by Newton's method, using the residuals
\eqref{eq:block_reduction_R0}--\eqref{eq:block_reduction_Ri}
and the Jacobian formulas
\eqref{eq:jacobian_R0_r}--\eqref{eq:jacobian_constraint_lambda}.

\State
Set
\[
    r_{\mathrm{g},h}^{n+1}=r,
    \qquad
    \lambda_{i,h}^{n+1}=\lambda_i,
    \quad i=1,\ldots,K.
\]

\State
Reconstruct the normal velocity
\[
    \mathbf{V}_{\nu}^{n+1}
    =
    \frac{r_{\mathrm{g},h}^{n+1}}{S_{\mathrm{g},h}^n}
    \mathbf{V}_{\mathrm{g},h}^n
    +
    \sum_{i=1}^K
    \lambda_{i,h}^{n+1}\mathbf{V}_{i,h}^n .
\]

\State
Update the nodes by
\[
    \mathbf{x}_j^{n+1}
    =
    \mathbf{x}_j^n
    +
    \Delta t
    \left(
        [\mathbf{V}_{\nu}^{n+1}]_j\bm{\nu}_j^n
        +
        [\mathbf{V}_{\tau}^{n+1}]_j\bm{\tau}_j^n
    \right),
    \qquad
    j=1,\ldots,N.
\]

\end{algorithmic}
\end{algorithm}
\FloatBarrier

Algorithm~\ref{alg:dual_sav_one_step} is an algebraic implementation of the
fully discrete system \eqref{eq:fully_discrete_system}. Therefore, once the
reduced nonlinear system is solved, the computed update satisfies the
modified-energy dissipation estimate in
Theorem~\ref{thm:fully_discrete_modified_energy_dissipation}.

\section{Applications and numerical validations}
\label{sec:numerical_validations}

We now test the fully discrete dual-SAV parametric finite element framework on
a selected set of planar geometric flows. The purpose of this section is not to
provide an exhaustive catalogue of geometric-flow simulations, but to verify the
main numerical features of the proposed method: temporal accuracy,
modified-energy dissipation, enforcement of nonlinear global constraints, mesh
quality improvement by tangential redistribution, and the practical behavior of
the algebraic block reduction.

The experiments are chosen so that each one tests a distinct component of the
framework. First, the curve shortening flow for a shrinking circle is used as a
baseline accuracy test, since an exact solution is available. Second, the
area-preserving curve shortening flow is used to examine the interaction
between tangential mesh redistribution and a single nonlinear global constraint.
Third, the curve diffusion flow is used to test the nonlocal $H^{-1}$ metric
within the same algebraic structure. Finally, a Helfrich-type flow with
simultaneous area and length constraints is used as a higher-order nonlinear
test with two global constraints.

Throughout this section, the stabilization terms are used in the zero-order
sense described in Section~\ref{sec:time_discretization}: they act on the new
velocity itself and do not propagate velocity history. Unless otherwise stated,
the artificial mesh regularization uses the uniform weight $w=1$. When a
comparison with purely Lagrangian motion is needed, we use the reference choice
$w=|\partial_\rho\bm{\gamma}|^{-1}$, for which no artificial tangential
redistribution is intended.

Unless explicitly stated otherwise, no post-step remeshing, no constraint
projection, and no SAV reinitialization are applied in the following
experiments.

\subsection{Experimental setup and diagnostic quantities}
\label{subsec:numerical_setup}

All computations are performed on periodic polygonal curves with $N$ vertices.
The nonlinear reduced system
\eqref{eq:block_reduction_reduced_system}
is solved by Newton's method with tolerance
$\varepsilon_{\mathrm{Newton}}$. The precise values of $N$, $\Delta t$,
the final time, and the stabilization parameters are specified in each
experiment.

We use the edge-ratio indicator
\[
    Q_{\mathrm{mesh}}^n
    :=
    \frac{\max_j l_j^n}{\min_j l_j^n},
    \qquad
    l_j^n:=|\mathbf{x}_{j+1}^n-\mathbf{x}_j^n|,
\]
to quantify the mesh quality. Values close to one indicate nearly uniform
edge lengths, while large values indicate vertex clustering or depletion.

For constrained computations, we report the relative area and length errors
\[
    e_A^n
    :=
    \frac{|A_h(\bm{\gamma}_h^n)-A_h(\bm{\gamma}_h^0)|}
         {|A_h(\bm{\gamma}_h^0)|},
    \qquad
    e_L^n
    :=
    \frac{|L_h(\bm{\gamma}_h^n)-L_h(\bm{\gamma}_h^0)|}
         {|L_h(\bm{\gamma}_h^0)|},
\]
whenever the corresponding constraints are imposed.

Since the theoretical estimates are stated for the modified SAV energies, we
distinguish these quantities from the original geometric energies. We denote
\[
    E_{\mathrm{SAV},g}^n:=(r_{\mathrm{g},h}^n)^2,
    \qquad
    E_{\mathrm{SAV},m}^n:=(r_{\mathrm{m},h}^n)^2.
\]
When the original geometric energy is also evaluated, we report the SAV gap
\[
    \operatorname{Gap}_{\mathrm{g}}^n
    :=
    \left|
        (r_{\mathrm{g},h}^n)^2
        -
        \left(
            E_{\mathrm{geom}}(\bm{\gamma}_h^n)+C_{\mathrm{g}}
        \right)
    \right|.
\]
The number of Newton iterations per time step is also recorded in the
constraint tests in order to assess the practical behavior of the reduced
algebraic system.

\subsection{Accuracy verification by the curve shortening flow}
\label{subsec:csf_accuracy}

\subsection{Accuracy verification by the curve shortening flow}

We first verify the temporal accuracy of the proposed scheme using the curve
shortening flow. This test is chosen because an exact solution is available for
a shrinking circle. We consider the length energy
\[
    E_{\mathrm{geom}}(\bm{\gamma})=L(\bm{\gamma}),
\]
whose scalar normal driving force is
\[
    w_{\mathrm{g}}=-\kappa.
\]
There are no global constraints in this test, so \(K=0\). The normal metric is
the \(L^2\) metric, and the mesh regularization uses the uniform weight \(w=1\).

As the initial curve, we take the unit circle centered at the origin. Under the
curve shortening flow, the exact solution remains a circle and its radius is
given by
\[
    R(t)=\sqrt{1-2t},
\]
up to the extinction time \(t=1/2\). We compute the solution up to
\(T=0.25\) using \(N=512\) vertices. The stabilization parameters are fixed as
\(\beta_\nu=10\), \(\beta_\tau=100\), and the SAV shifts are
\(C_{\mathrm{g}}=C_{\mathrm{m}}=1\). The time step sizes are
\[
    \Delta t
    =
    10^{-3},\ 5\times10^{-4},\ 2.5\times10^{-4},\ 1.25\times10^{-4}.
\]
The error is measured by the maximum error in the nodal radii:
\[
    e_{\Delta t}
    :=
    \max_{1\leq j\leq N}
    \left|
        |\mathbf{x}_j^M|-R(T)
    \right|,
    \qquad
    T=M\Delta t.
\]
The observed order of convergence is computed by
\[
    \mathrm{EOC}
    =
    \frac{
        \log(e_{\Delta t}/e_{\Delta t/2})
    }{
        \log 2
    }.
\]

Table~\ref{tab:csf_time_convergence} reports the temporal convergence results.
The observed convergence rate is approximately first order, which is consistent
with the semi-implicit backward-Euler structure of the scheme. The mesh-quality
indicator remains equal to one in this radially symmetric test, and the SAV gap
decreases proportionally to the time step.

\begin{table}[htbp]
\centering
\caption{Temporal convergence for the shrinking circle under the curve
shortening flow. The parameters are $N=512$, $T=0.25$,
$\beta_\nu=10$, $\beta_\tau=100$, and
$C_{\mathrm{g}}=C_{\mathrm{m}}=1$.}
\label{tab:csf_time_convergence}
\begin{tabular}{c c c c c}
\hline
$\Delta t$ & $e_{\Delta t}$ & EOC & $Q_{\mathrm{mesh}}$ & $\operatorname{Gap}_{\mathrm{g}}$ \\
\hline
$1.0\times 10^{-3}$  & $4.432782\times 10^{-4}$ & --     & $1.000000$ & $5.083406\times 10^{-4}$ \\
$5.0\times 10^{-4}$  & $2.219714\times 10^{-4}$ & $0.9978$ & $1.000000$ & $2.546265\times 10^{-4}$ \\
$2.5\times 10^{-4}$  & $1.110690\times 10^{-4}$ & $0.9989$ & $1.000000$ & $1.274277\times 10^{-4}$ \\
$1.25\times 10^{-4}$ & $5.555537\times 10^{-5}$ & $0.9995$ & $1.000000$ & $6.374248\times 10^{-5}$ \\
\hline
\end{tabular}
\end{table}

\subsection{Mesh redistribution and single-constraint enforcement}
\label{subsec:apcsf_mesh_constraint}

We next consider the area-preserving curve shortening flow. This example is
used to test the coupling between tangential mesh redistribution and a single
nonlinear global constraint. The geometric energy is the perimeter
\[
    E_{\mathrm{geom}}(\bm{\gamma})=L(\bm{\gamma}),
\]
and the enclosed area is prescribed to be conserved:
\[
    A_h(\bm{\gamma}_h^n)=A_h(\bm{\gamma}_h^0).
\]
Thus, this test corresponds to one global constraint, namely $K=1$.

The normal driving force is
\[
    w_{\mathrm{g}}=-\kappa ,
\]
and the normal metric is the $L^2$ metric. We compare two parametrization
strategies. The first is the Lagrangian reference case,
\[
    w=|\partial_\rho\bm{\gamma}|^{-1},
\]
for which no artificial tangential redistribution is intended. The second is
the uniform mesh redistribution used in the proposed method,
\[
    w=1.
\]
This comparison is designed to isolate the role of the tangential mesh
equation: both computations use the same normal geometric law and the same
area constraint, but only the second computation actively redistributes mesh
points.

The initial curve is the non-convex star-shaped curve
\[
    \bm{\gamma}(\theta,0)
    =
    r(\theta)
    \begin{pmatrix}
        \cos\theta\\
        \sin\theta
    \end{pmatrix},
    \qquad
    r(\theta)=1+0.9\cos(5\theta),
    \qquad
    0\leq\theta<2\pi .
\]
We use
\[
    N=256,\qquad
    T=0.50,\qquad
    \Delta t=5.0\times 10^{-4}.
\]
The stabilization parameters and SAV shifts are
\[
    \beta_\nu=10,\qquad
    \beta_\tau=100,\qquad
    C_{\mathrm{g}}=C_{\mathrm{m}}=1.
\]
The reduced Newton system is solved with tolerance
\[
    \varepsilon_{\mathrm{Newton}}=10^{-12},
\]
and the maximum number of Newton iterations is set to $20$.

Figure~\ref{fig:apcsf_mesh_comparison} shows the curve evolution for the two
parametrization strategies. In the Lagrangian reference computation, the
vertices become severely clustered during the evolution. Although the
algebraic system can be advanced to the final time and the area constraint is
still enforced, the resulting polygonal curve is no longer a reliable
approximation of the intended area-preserving geometric relaxation. We
therefore regard this run as a mesh-degeneration baseline rather than as a
successful computation. In contrast, the uniform tangential redistribution
prevents the severe vertex clustering and produces a smooth area-preserving
relaxation with a nearly uniform final vertex distribution.

\begin{figure}[htbp]
    \centering
    \includegraphics[width=0.95\textwidth]{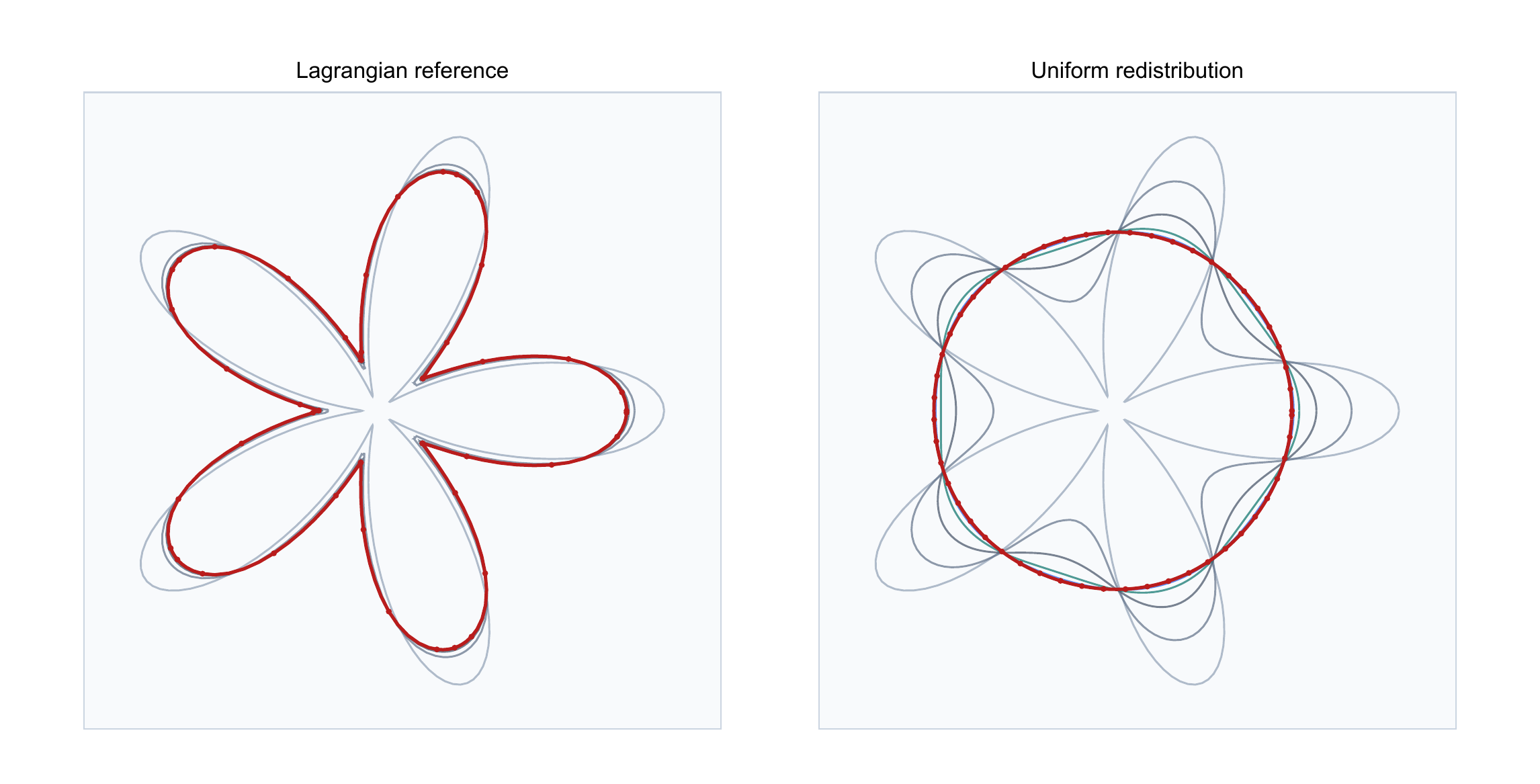}
    \caption{Area-preserving curve shortening flow from the star-shaped initial
    curve $r(\theta)=1+0.9\cos(5\theta)$. Left: Lagrangian reference
    parametrization, where severe vertex clustering makes the final polygonal
    curve unreliable as a geometric approximation. Right: uniform tangential
    mesh redistribution, which maintains a regular vertex distribution during
    the relaxation.}
    \label{fig:apcsf_mesh_comparison}
\end{figure}

Figure~\ref{fig:apcsf_mesh_ratio} shows the time history of the mesh-ratio
indicator $Q_{\mathrm{mesh}}^n$. The Lagrangian reference parametrization
develops rapidly growing mesh distortion. The uniform redistribution also
experiences an initial transient caused by the strongly non-convex initial
shape, but the tangential mesh equation subsequently regularizes the vertex
distribution and drives the mesh ratio back close to one.

\begin{figure}[htbp]
    \centering
    \includegraphics[width=0.65\textwidth]{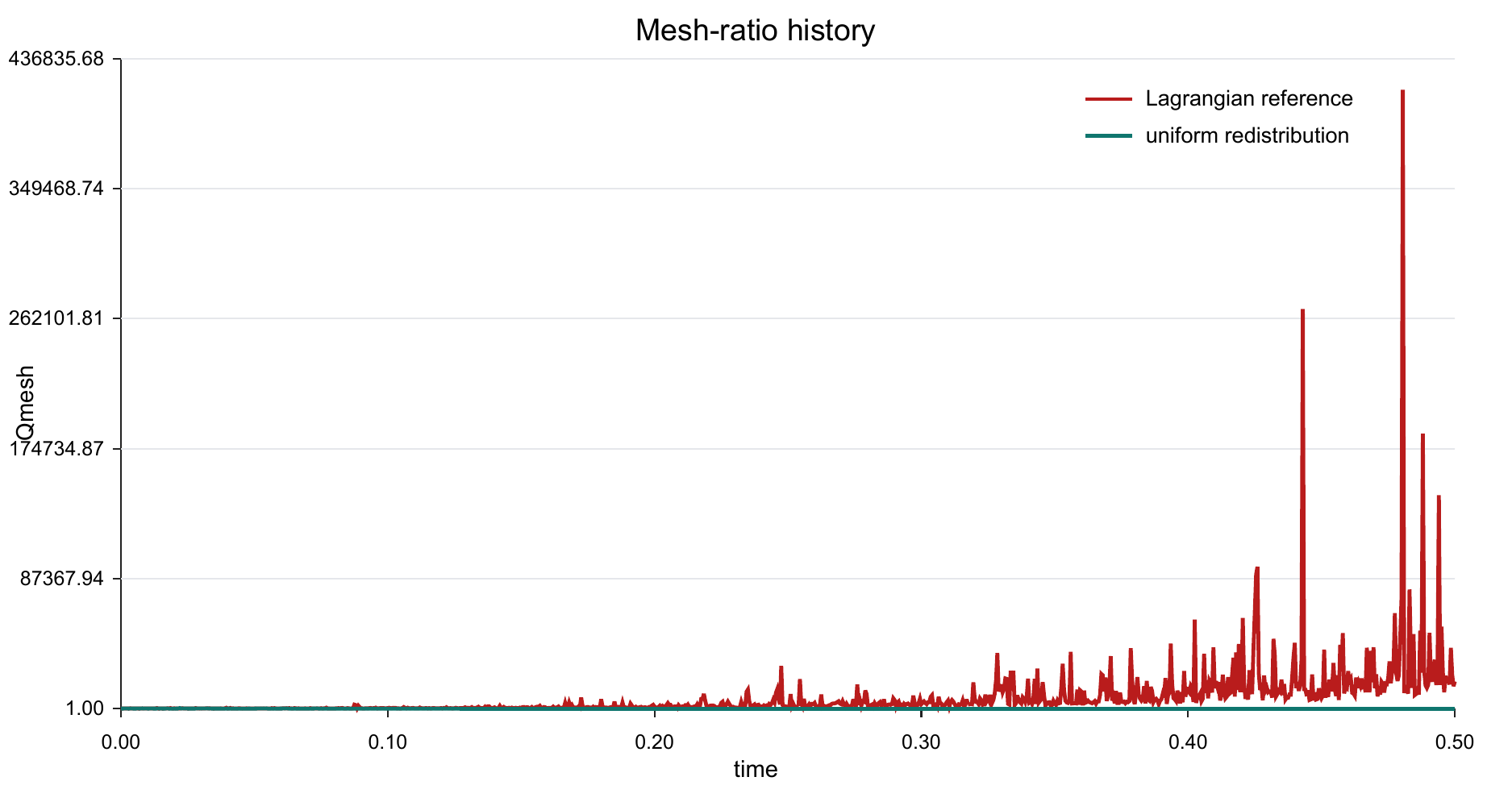}
    \caption{Time history of the mesh-ratio indicator
    $Q_{\mathrm{mesh}}^n$ for the area-preserving curve shortening flow.}
    \label{fig:apcsf_mesh_ratio}
\end{figure}

The same behavior is quantified in
Table~\ref{tab:apcsf_mesh_comparison}. The Lagrangian reference run preserves
the area constraint through the reduced Newton solve, but this does not prevent
mesh degeneration: the maximum mesh ratio reaches $4.16\times10^5$, and the
minimum edge length decreases to $2.79\times10^{-7}$. By contrast, the uniform
tangential redistribution keeps the final mesh nearly equidistributed, with
$Q_{\mathrm{mesh}}^M=1.0008$, while preserving the area constraint to
approximately $10^{-15}$ relative accuracy. The average number of Newton
iterations remains small in both cases, indicating that the reduced nonlinear
system is not the source of the Lagrangian failure; the failure is caused by
the degeneration of the parametrization.

\begin{table}[htbp]
\centering
\small
\caption{Mesh quality, area preservation, and Newton statistics for the
area-preserving curve shortening flow. Here
\(\max Q=\max_n Q_{\mathrm{mesh}}^n\),
\(Q^M=Q_{\mathrm{mesh}}^M\),
\(\min l=\min_{n,j}l_j^n\), and
\(\max e_A=\max_n e_A^n\). The parameters are
\(N=256\), \(T=0.50\), \(\Delta t=5.0\times10^{-4}\),
\(\beta_\nu=10\), \(\beta_\tau=100\), \(C_{\mathrm g}=C_{\mathrm m}=1\), and
\(\varepsilon_{\mathrm{Newton}}=10^{-12}\).}
\label{tab:apcsf_mesh_comparison}
\begin{tabular}{@{}lccccc@{}}
\hline
strategy
&
$\max Q$
&
$Q^M$
&
$\min l$
&
$\max e_A$
&
avg. iter. \\
\hline
Lagrangian
&
$4.16{\times}10^{5}$
&
$1.83{\times}10^{4}$
&
$2.79{\times}10^{-7}$
&
$2.26{\times}10^{-14}$
&
$3.43$
\\
uniform
&
$3.52{\times}10^{2}$
&
$1.0008$
&
$2.24{\times}10^{-4}$
&
$1.82{\times}10^{-15}$
&
$3.09$
\\
\hline
\end{tabular}
\end{table}

\subsection{Nonlocal metric test: curve diffusion flow}
\label{subsec:cdf_test}

We next consider the curve diffusion flow as a test of the nonlocal
$H^{-1}$ metric. The purpose of this experiment is to verify that the proposed
dual-SAV PFEM can incorporate an $H^{-1}$ normal response within the same
algebraic framework as the preceding $L^2$-metric examples.

The geometric energy is the perimeter
\[
    E_{\mathrm{geom}}(\bm{\gamma})=L(\bm{\gamma}),
\]
and the corresponding scalar geometric force is
\[
    w_{\mathrm{g}}=-\kappa .
\]
The normal metric is the $H^{-1}$ metric on the discrete zero-mean subspace.
Thus, no area constraint is imposed by a Lagrange multiplier in this experiment.
Instead, the enclosed area is monitored as a diagnostic quantity. This reflects
the continuous structure of the curve diffusion flow, where area preservation
is encoded in the zero-mean character of the $H^{-1}$ normal velocity.

For this nonlocal fourth-order flow, the normal metric matrix is realized by
the discrete inverse Laplace--Beltrami operator on the zero-mean subspace,
\[
    \mathbf{A}_{\nu,h}^{n}
    =
    \mathbf{M}_{L,h}^{n}
    (\mathbf{K}_h^n)^\dagger
    \mathbf{M}_{L,h}^{n}.
\]
In the implementation, this inverse is computed by a rank-one completion of
the stiffness matrix together with the lumped-mass zero-mean constraint. 
For compatible zero-mean right-hand sides, this rank-one completion gives the
same response as the Moore--Penrose inverse restricted to the discrete
zero-mean subspace.
The normal load vector is projected onto the dual of the lumped-mass zero-mean
subspace before solving the response problem.

We use the hybrid normal stabilization
\[
    \mathbf{D}_{\nu,h}^n
    =
    \beta_{\nu,4}
    \mathbf{K}_h^n
    (\mathbf{M}_{\mathrm{L},h}^n)^{-1}
    \mathbf{K}_h^n
    +
    \beta_{\nu,2}
    \mathbf{K}_h^n .
\]
Both terms are symmetric positive semi-definite. Hence this stabilization is
covered by the assumptions used in
Theorem~\ref{thm:fully_discrete_modified_energy_dissipation}. The biharmonic
component damps high-frequency stiffness, while the Laplacian component damps
intermediate-frequency oscillations in the $H^{-1}$ response.

The length-force vector is evaluated on the frozen polygonal curve and is used
consistently in both the normal velocity equation and the geometric SAV update
equation. No Helmholtz force filtering is used in this experiment. During the time integration, no post-step remeshing, no area
projection, and no SAV reinitialization are applied. The only redistribution
outside the time stepping is the initial arclength redistribution of the
vertices.

The initial curve is
\[
    \bm{\gamma}(\theta,0)
    =
    r(\theta)
    \begin{pmatrix}
        \cos\theta\\
        \sin\theta
    \end{pmatrix},
    \quad
    r(\theta)
    =
    1+0.2\cos(3\theta)+0.1\sin(5\theta),
    \qquad
    0\leq \theta<2\pi .
\]
Before time integration, the initial vertices are redistributed uniformly with
respect to arclength. We use
\[
    N=256,
    \qquad
    \Delta t=10^{-5},
    \qquad
    T=0.10 .
\]
The stabilization parameters are
\[
    \beta_{\nu,4}=10,
    \qquad
    \beta_{\nu,2}=10,
    \qquad
    \beta_{\tau}=10,
\]
and the SAV shifts are
\[
    C_{\mathrm{g}}=C_{\mathrm{m}}=1 .
\]

Figure~\ref{fig:cdf_evolution} shows the evolution of the curve. The curve
relaxes toward a more circular configuration while maintaining a regular
polygonal representation. The final vertices are displayed to indicate the
mesh distribution at the final time.

\begin{figure}[htbp]
    \centering
    \includegraphics[width=0.5\textwidth]{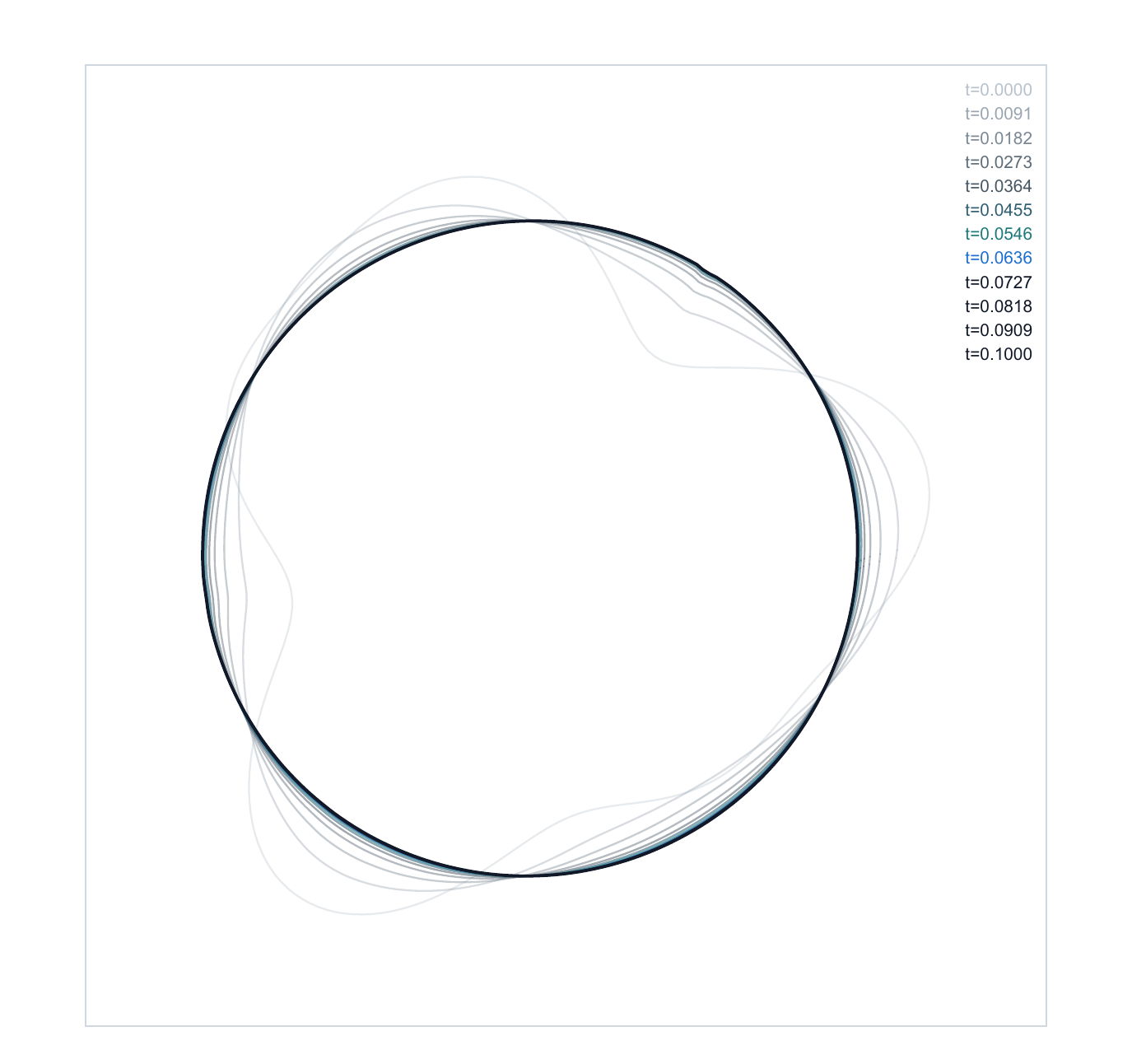}
    \caption{Curve diffusion flow computed with the $H^{-1}$ normal metric and
    hybrid normal stabilization. No post-step remeshing, area projection, SAV
    reinitialization, or Helmholtz force filtering is applied during the time
    integration. Dots indicate final-time vertices.}
    \label{fig:cdf_evolution}
\end{figure}

Table~\ref{tab:cdf_diagnostics} summarizes the quantitative diagnostics for
this computation. The area error remains small, with
\(\max_n e_A^n=7.910\times10^{-6}\), although no area constraint is imposed
through a Lagrange multiplier. This confirms that the zero-mean \(H^{-1}\) response keeps the area error small
in this test. The mesh
quality is also well controlled: the maximum mesh ratio is
\(\max_n Q_{\mathrm{mesh}}^n=1.2468\), and the final mesh ratio is
\(Q_{\mathrm{mesh}}^M=1.0254\). The final isoperimetric deficit is
\(\delta_{\mathrm{iso}}(T)=2.7644\times10^{-4}\), which is consistent with
relaxation toward a nearly circular steady state. The final geometric SAV gap
is \(\operatorname{Gap}_{\mathrm{g}}(T)=3.939\times10^{-4}\).

\begin{table}[htbp]
\centering
\small
\caption{Diagnostics for the curve diffusion flow. Here
\(\max e_A=\max_n e_A^n\),
\(\max Q=\max_n Q_{\mathrm{mesh}}^n\),
\(Q^M=Q_{\mathrm{mesh}}^M\),
\(\delta_{\mathrm{iso}}(T)=\delta_{\mathrm{iso}}^M\), and
\(\operatorname{Gap}_{\mathrm{g}}(T)=\operatorname{Gap}_{\mathrm{g}}^M\).}
\label{tab:cdf_diagnostics}
\begin{tabular}{@{}cccccccc@{}}
\hline
\(N\)
&
\(\Delta t\)
&
\(T\)
&
\(\max e_A\)
&
\(\max Q\)
&
\(Q^M\)
&
\(\delta_{\mathrm{iso}}(T)\)
&
\(\operatorname{Gap}_{\mathrm{g}}(T)\)
\\
\hline
\(256\)
&
\(1.0{\times}10^{-5}\)
&
\(0.10\)
&
\(7.910{\times}10^{-6}\)
&
\(1.2468\)
&
\(1.0254\)
&
\(2.7644{\times}10^{-4}\)
&
\(3.939{\times}10^{-4}\)
\\
\hline
\end{tabular}
\end{table}

\subsection{Helfrich-type flow with area and length constraints}
\label{subsec:helfrich_test}

We finally consider a Helfrich-type flow with simultaneous area and length
constraints. This experiment is designed as a nonlinear fourth-order test of
the proposed dual-SAV PFEM. In contrast to the preceding examples, the
geometric energy contains a nonlinear curvature term, and two global geometric
quantities are prescribed. Thus, this test corresponds to $K=2$, and the
reduced nonlinear system has dimension three, independently of the number of
mesh vertices.

We use the Helfrich-type bending energy
\[
    E_{\mathrm{geom}}(\bm{\gamma})
    =
    \frac12
    \int_{\bm{\gamma}}
    (\kappa-c_0)^2\,\mathrm{d}s,
\]
where $c_0$ is the spontaneous curvature parameter. With the sign convention used in this paper, the scalar normal force used in
the computation is the frozen polygonal counterpart of
\[
    w_{\mathrm{g}}
    =
    \partial_s^2\kappa
    +
    \frac12\kappa(\kappa^2-c_0^2).
\]
The linear term in \(\kappa\) in the energy does not contribute to the first
variation for closed planar curves with fixed turning number. In the implementation,
the corresponding normal force is evaluated on the frozen polygonal curve.

The enclosed area and the perimeter are constrained to their initial values:
\[
    A_h(\bm{\gamma}_h^n)=A_h(\bm{\gamma}_h^0),
    \qquad
    L_h(\bm{\gamma}_h^n)=L_h(\bm{\gamma}_h^0).
\]
The corresponding Lagrange multipliers are denoted by
$\lambda_{A,h}^{n+1}$ and $\lambda_{L,h}^{n+1}$. The normal metric is the
$L^2$ metric. Hence, unlike the curve diffusion test in
Section~\ref{subsec:cdf_test}, no zero-mean restriction is imposed on the
normal response space.

The normal stabilization is chosen as the biharmonic-type stabilization
\[
    \mathbf{D}_{\nu,h}^n
    =
    \beta_\nu
    \mathbf{K}_h^n
    (\mathbf{M}_{\mathrm{L},h}^n)^{-1}
    \mathbf{K}_h^n ,
\]
where $\mathbf{K}_h^n$ is the stiffness matrix and $\mathbf{M}_{\mathrm{L},h}^n$ is the
lumped mass matrix on the frozen polygonal curve. In this subsection, we write \(\beta_\nu\) for the biharmonic stabilization
parameter \(\beta_{\nu,4}\), since no additional Laplacian component is used in
the normal stabilization. This matrix is symmetric
positive semi-definite, and is therefore covered by the abstract stabilization
assumption used in
Theorem~\ref{thm:fully_discrete_modified_energy_dissipation}. The tangential
mesh equation uses the $L^2$ metric and the Laplacian stabilization
$\beta_\tau \mathbf{K}_h^n$.

At each time step, we first solve the three normal response problems associated
with the bending force and the two constraint gradients. The nonlinear part is
then reduced to the scalar unknowns
\[
    \left(
        r_{\mathrm{g},h}^{n+1},
        \lambda_{A,h}^{n+1},
        \lambda_{L,h}^{n+1}
    \right).
\]
Thus, the Newton iteration is performed on a three-dimensional system, while
the large spatial solves are only linear response problems.

The initial curve is the ellipse
\[
    \bm{\gamma}(\theta,0)
    =
    \begin{pmatrix}
        4\cos\theta\\
        \sin\theta
    \end{pmatrix},
    \qquad
    0\leq\theta<2\pi .
\]
We use
\[
    N=256,
    \qquad
    \Delta t=10^{-4},
    \qquad
    T=0.50 .
\]
The Helfrich and numerical parameters are
\[
    c_0=0.5,
    \qquad
    \beta_\nu=10,
    \qquad
    \beta_\tau=10,
    \qquad
    C_{\mathrm{g}}=C_{\mathrm{m}}=1 .
\]
The Newton tolerance is
\[
    \varepsilon_{\mathrm{Newton}}=10^{-10}.
\]
No post-step remeshing, no constraint projection, and no SAV reinitialization
are applied during the computation.

Figure~\ref{fig:helfrich_evolution} shows the evolution of the curve. The
initially elongated ellipse relaxes under the Helfrich-type bending energy,
while the enclosed area and perimeter are constrained to remain equal to their
initial values. The dots on the final curve indicate the final-time vertices.

\begin{figure}[htbp]
    \centering
    \includegraphics[width=0.5\textwidth]{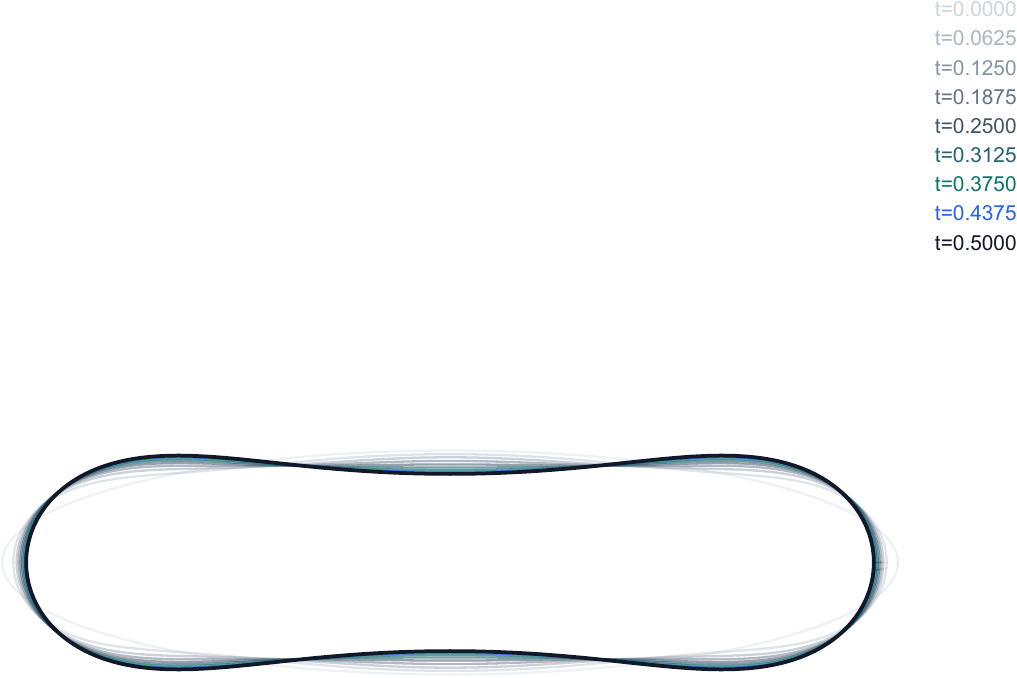}
    \caption{Helfrich-type flow with simultaneous area and length constraints.
    The evolution is driven by the bending energy
    $\frac12\int(\kappa-c_0)^2\,\mathrm{d}s$, while the enclosed area and perimeter are
    prescribed to remain equal to their initial values. Dots indicate
    final-time vertices.}
    \label{fig:helfrich_evolution}
\end{figure}

We monitor the relative area and length errors
\[
    e_A^n
    =
    \frac{
        |A_h(\bm{\gamma}_h^n)-A_h(\bm{\gamma}_h^0)|
    }{
        |A_h(\bm{\gamma}_h^0)|
    },
    \qquad
    e_L^n
    =
    \frac{
        |L_h(\bm{\gamma}_h^n)-L_h(\bm{\gamma}_h^0)|
    }{
        |L_h(\bm{\gamma}_h^0)|
    } .
\]
We also monitor the mesh-ratio indicator
\[
    Q_{\mathrm{mesh}}^n
    =
    \frac{\max_j l_j^n}{\min_j l_j^n},
\]
the original discrete bending energy
\[
    E_{\mathrm{bend},h}^n
    =
    \frac12
    \int_{\bm{\gamma}_h^n}
    (\kappa_h^n-c_0)^2\,\mathrm{d}s,
\]
and the modified geometric SAV energy
\[
    E_{\mathrm{SAV},g}^n
    =
    (r_{\mathrm{g},h}^n)^2 .
\]
Since the theoretical dissipation estimate concerns the modified SAV energy,
the original bending energy is reported as a physical diagnostic quantity. We
also report the geometric SAV gap
\[
    \operatorname{Gap}_{\mathrm{g}}^n
    =
    \left|
        (r_{\mathrm{g},h}^{n})^2
        -
        \left(
            E_{\mathrm{bend},h}^{n}+C_{\mathrm{g}}
        \right)
    \right|.
\]

Figure~\ref{fig:helfrich_diagnostics} shows the diagnostic histories. The
modified SAV energy decreases monotonically, in agreement with
Theorem~\ref{thm:fully_discrete_modified_energy_dissipation}. The original
bending energy also decreases in this computation, from
\[
    E_{\mathrm{bend},h}^0=4.6032
    \quad\text{to}\quad
    E_{\mathrm{bend},h}^M=2.0745 .
\]
The area and length errors remain at the level of round-off errors, indicating
that the reduced Newton system enforces both global constraints accurately.
The mesh-ratio indicator decreases substantially from the initially elongated
configuration and remains close to one at the final time.

\begin{figure}[htbp]
    \centering
    \includegraphics[width=0.78\textwidth]{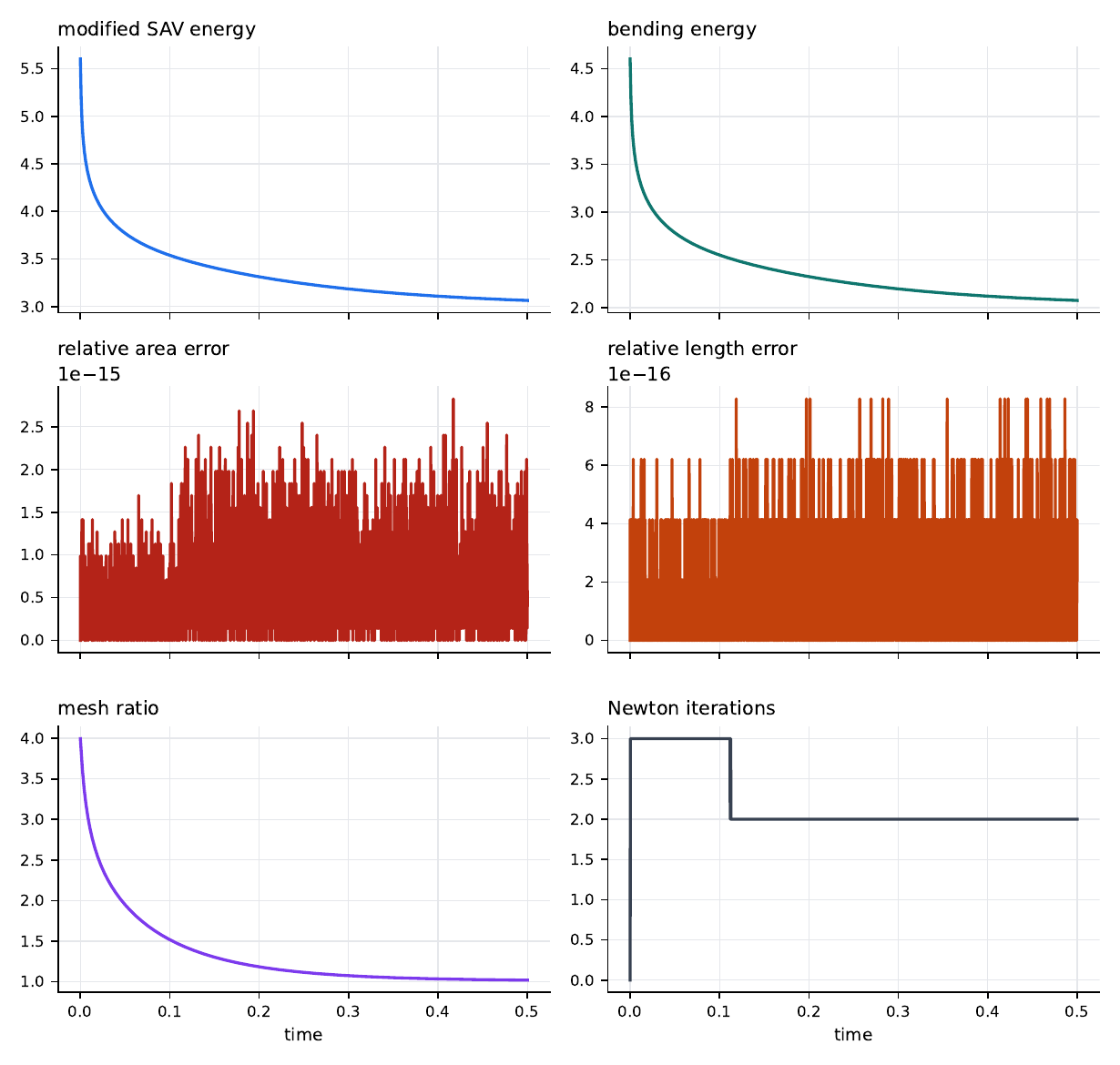}
    \caption{Diagnostic histories for the Helfrich-type flow. The plot shows
    the modified SAV energy, original bending energy, relative area error,
    relative length error, mesh-ratio indicator, and Newton iteration count.}
    \label{fig:helfrich_diagnostics}
\end{figure}

Table~\ref{tab:helfrich_diagnostics} summarizes the quantitative diagnostics.
The area and length constraints are preserved up to errors close to round-off level,
with
\[
    \max_n e_A^n=2.827\times10^{-15},
    \qquad
    \max_n e_L^n=8.283\times10^{-16}.
\]
The maximum mesh ratio is
\[
    \max_n Q_{\mathrm{mesh}}^n=3.9952,
\]
which reflects the initially elongated ellipse. During the computation, the
tangential mesh redistribution improves the vertex distribution, and the final
mesh ratio decreases to
\[
    Q_{\mathrm{mesh}}^M=1.0197 .
\]
The average number of Newton iterations is \(2.22\), confirming that the three-dimensional reduced nonlinear solve adds only a modest nonlinear
algebraic cost. The final
geometric SAV gap is
\[
    \operatorname{Gap}_{\mathrm{g}}(T)=1.022\times10^{-2}.
\]

\begin{table}[htbp]
\centering
\small
\caption{Diagnostics for the Helfrich-type flow with simultaneous area and
length constraints. Here
\(\max e_A=\max_n e_A^n\),
\(\max e_L=\max_n e_L^n\),
\(\max Q=\max_n Q_{\mathrm{mesh}}^n\),
\(Q^M=Q_{\mathrm{mesh}}^M\), and
\(\operatorname{Gap}_{\mathrm{g}}(T)=\operatorname{Gap}_{\mathrm{g}}^M\).}
\label{tab:helfrich_diagnostics}
\begin{tabular}{@{}cccccccc@{}}
\hline
\(N\)
&
\(\Delta t\)
&
\(\max e_A\)
&
\(\max e_L\)
&
\(\max Q\)
&
\(Q^M\)
&
avg. iter.
&
\(\operatorname{Gap}_{\mathrm{g}}(T)\)
\\
\hline
\(256\)
&
\(1.0{\times}10^{-4}\)
&
\(2.827{\times}10^{-15}\)
&
\(8.283{\times}10^{-16}\)
&
\(3.9952\)
&
\(1.0197\)
&
\(2.22\)
&
\(1.022{\times}10^{-2}\)
\\
\hline
\end{tabular}
\end{table}

To further illustrate the computational implication of the block reduction, we
repeat the Helfrich-type experiment with different numbers of vertices, using
the same parameter values as above except for \(N\). For
this problem, the number of global constraints is \(K=2\). Hence, after the
linear response problems have been solved, the remaining nonlinear system is
always three-dimensional, independently of \(N\). Table~\ref{tab:helfrich_cost}
reports the number of response solves per time step, the average number of
Newton iterations, the average CPU time per time step, and the maximum
constraint errors.

For each time step, four response problems are solved: one for the bending
force, two for the area and length constraint gradients, and one for the
tangential mesh equation. The average number of Newton iterations remains
essentially unchanged as \(N\) increases. Thus, the increase in the CPU time is
mainly due to the larger linear response problems, while the nonlinear
algebraic part remains low-dimensional.

\begin{table}[htbp]
\centering
\footnotesize
\setlength{\tabcolsep}{4pt}
\caption{Computational cost for the Helfrich-type flow with simultaneous area
and length constraints. The reduced nonlinear system has dimension three for
all values of \(N\). Here ``resp.'' denotes the number of linear response
solves per time step, and CPU times are reported in seconds per time step
after excluding Julia compilation warm-up. The time step size is
\(\Delta t=1.0\times10^{-4}\) in all cases.}
\label{tab:helfrich_cost}
\begin{tabular}{@{}cccccc@{}}
\hline
\(N\)
&
resp.
&
avg. iter.
&
CPU/step (s)
&
\(\max e_A\)
&
\(\max e_L\)
\\
\hline
\(128\)
&
\(4\)
&
\(2.23\)
&
\(4.264{\times}10^{-4}\)
&
\(2.121{\times}10^{-15}\)
&
\(8.284{\times}10^{-16}\)
\\
\(256\)
&
\(4\)
&
\(2.22\)
&
\(1.645{\times}10^{-3}\)
&
\(2.827{\times}10^{-15}\)
&
\(8.283{\times}10^{-16}\)
\\
\(512\)
&
\(4\)
&
\(2.22\)
&
\(8.915{\times}10^{-3}\)
&
\(4.382{\times}10^{-15}\)
&
\(1.035{\times}10^{-15}\)
\\
\hline
\end{tabular}
\end{table}

The constraint errors remain close to round-off level for all tested mesh
sizes. In particular, the maximum area and length errors remain below
\(5\times10^{-15}\) and \(2\times10^{-15}\), respectively. The average Newton
iteration count stays near \(2.22\) for all values of \(N\). This indicates that
the reduced nonlinear solve does not become more difficult as the spatial
resolution is increased in this test.

Finally, we examine the relation between the modified geometric SAV energy and
the original bending energy for the Helfrich-type flow. The dissipation
estimate in Theorem~\ref{thm:fully_discrete_modified_energy_dissipation}
concerns the modified quantity
\[
    (r_{\mathrm{g},h}^n)^2,
\]
whereas the original bending energy is
\[
    E_{\mathrm{bend},h}^n
    =
    \frac12
    \int_{\bm{\gamma}_h^n}
    (\kappa_h^n-c_0)^2\,\mathrm{d}s .
\]
We therefore measure the geometric SAV gap
\[
    \operatorname{Gap}_{\mathrm{g}}^n
    =
    \left|
        (r_{\mathrm{g},h}^{n})^2
        -
        \left(
            E_{\mathrm{bend},h}^{n}
            +
            C_{\mathrm g}
        \right)
    \right|.
\]
For this test, we use the same Helfrich-type setting as above and vary only
the time step size. We also report the relative SAV gap
\[
    \operatorname{RelGap}_{\mathrm{g}}(T)
    =
    \frac{
        \operatorname{Gap}_{\mathrm{g}}(T)
    }{
        E_{\mathrm{bend},h}^M+C_{\mathrm g}
    },
\]
where \(M\) denotes the final time step.

\begin{table}[!htbp]
\centering
\small
\caption{Time-step dependence of the geometric SAV gap for the Helfrich-type
flow. Here
\(\operatorname{Gap}_{\mathrm{g}}(T)=\operatorname{Gap}_{\mathrm{g}}^M\).
The parameters are the same as in the main Helfrich-type experiment, except
for \(\Delta t\).}
\label{tab:helfrich_sav_gap_dt}
\begin{tabular}{@{}cccccc@{}}
\hline
\(\Delta t\)
&
\(\operatorname{Gap}_{\mathrm{g}}(T)\)
&
\(\operatorname{RelGap}_{\mathrm{g}}(T)\)
&
\(\max e_A\)
&
\(\max e_L\)
&
avg. iter.
\\
\hline
\(2.0{\times}10^{-4}\)
&
\(1.406{\times}10^{-2}\)
&
\(4.571{\times}10^{-3}\)
&
\(2.545{\times}10^{-15}\)
&
\(1.035{\times}10^{-15}\)
&
\(2.46\)
\\
\(1.0{\times}10^{-4}\)
&
\(1.022{\times}10^{-2}\)
&
\(3.324{\times}10^{-3}\)
&
\(2.827{\times}10^{-15}\)
&
\(8.283{\times}10^{-16}\)
&
\(2.22\)
\\
\(5.0{\times}10^{-5}\)
&
\(7.145{\times}10^{-3}\)
&
\(2.324{\times}10^{-3}\)
&
\(3.110{\times}10^{-15}\)
&
\(1.035{\times}10^{-15}\)
&
\(2.11\)
\\
\hline
\end{tabular}
\end{table}

Table~\ref{tab:helfrich_sav_gap_dt} shows that the geometric SAV gap decreases
as the time step is refined. The relative gap exhibits the same decreasing
trend. At the same time, the area and length constraints remain enforced up to
errors close to round-off level for all tested time step sizes.

This behavior is consistent with the interpretation that the SAV variable
tracks the shifted original bending energy more closely for smaller time step
sizes. Thus, although the theoretical dissipation estimate is stated for the
modified SAV energy, the numerical results indicate that the modified
geometric SAV energy remains close to the shifted original bending energy in
this Helfrich-type test. The average Newton iteration count also remains small,
indicating that time-step refinement does not make the reduced
three-dimensional nonlinear solve more difficult in this experiment.

\section{Conclusion}
\label{sec:conclusion}

We have developed a stabilized dual-SAV parametric finite element framework for
constrained planar geometric flows. The key idea is to separate the physical
geometric energy from the artificial mesh regularization energy by introducing
two scalar auxiliary variables. This separation allows the physical normal
evolution and the tangential mesh redistribution to be treated within a common
modified-energy framework, while preventing the artificial mesh mechanism from
introducing an additional normal driving force.

The semi-implicit frozen-metric discretization leads to linear spatial response
problems and satisfies modified-energy dissipation estimates for both the
geometric and mesh SAV variables. Nonlinear global constraints are handled by
an algebraic block reduction: after the response problems are solved, the
remaining nonlinear system involves only the geometric SAV variable and the
Lagrange multipliers. For \(K\) constraints, this system has dimension
\(K+1\), independently of the number of vertices.

The numerical experiments demonstrate the main features of the proposed
framework. The curve shortening test confirms first-order temporal accuracy,
the area-preserving curve shortening test illustrates the effectiveness of the
tangential mesh redistribution and the single-constraint block reduction, the
curve diffusion test shows that the nonlocal \(H^{-1}\) metric can be handled
within the same algebraic structure, and the Helfrich-type test verifies the
practical behavior of the three-dimensional reduced Newton system for
simultaneous area and length constraints.

Several issues remain for future work. First, the present analysis is restricted
to closed planar curves. Extending the full framework to evolving surfaces in
\(\mathbb{R}^3\) requires additional surface-mesh quality control and a separate
robustness analysis. Second, the relation between the modified SAV energies and
the original geometric energies deserves further investigation, especially for
larger time steps and higher-order flows. Third, while the present experiments show the usefulness of the algebraic block
reduction, a more systematic comparison with fully implicit monolithic and
discrete-gradient-type structure-preserving schemes remains an important topic
for future study.
A rigorous error analysis and a systematic parameter selection strategy for
the stabilization coefficients are also left for future work.

\bibliographystyle{siam}
\bibliography{2026_S}

\begin{thebibliography}{10}

\bibitem{bai2024convergent}
{\sc G.~Bai, J.~Hu, and B.~Li}, {\em A convergent evolving finite element method with artificial tangential motion for surface evolution under a prescribed velocity field}, SIAM J. Numer. Anal., 62 (2024), pp.~2172--2195.

\bibitem{bao2025energy}
{\sc W.~Bao, Y.~Li, and D.~Wang}, {\em An energy-stable parametric finite element method for the willmore flow in three dimensions}, 2025.

\bibitem{barrett2007parametric}
{\sc J.~W. Barrett, H.~Garcke, and R.~Nürnberg}, {\em A parametric finite element method for fourth order geometric evolution equations}, Journal of Computational Physics, 222 (2007), pp.~441--467.

\bibitem{barrett2008parametric}
\leavevmode\vrule height 2pt depth -1.6pt width 23pt, {\em On the parametric finite element approximation of evolving hypersurfaces in r3}, Journal of Computational Physics, 227 (2008), pp.~4281--4307.

\bibitem{barrett2020parametric}
\leavevmode\vrule height 2pt depth -1.6pt width 23pt, {\em Chapter 4 - parametric finite element approximations of curvature-driven interface evolutions}, in Geometric Partial Differential Equations - Part I, A.~Bonito and R.~H. Nochetto, eds., vol.~21 of Handbook of Numerical Analysis, Elsevier, 2020, pp.~275--423.

\bibitem{cheng2018multiple}
{\sc Q.~Cheng and J.~Shen}, {\em Multiple scalar auxiliary variable ({MSAV}) approach and its application to the phase-field vesicle membrane model}, SIAM J. Sci. Comput., 40 (2018), pp.~A3982--A4006.

\bibitem{deckelnick2005computation}
{\sc K.~Deckelnick, G.~Dziuk, and C.~M. Elliott}, {\em Computation of geometric partial differential equations and mean curvature flow}, Acta Numer., 14 (2005), pp.~139--232.

\bibitem{duan2026adaptive}
{\sc Z.~Duan and M.~Li}, {\em Adaptive finite difference methods for the willmore flow: mesh redistribution algorithm and tangential velocity approach}, 2026.

\bibitem{dziuk1991algorithm}
{\sc G.~Dziuk}, {\em An algorithm for evolutionary surfaces}, Numer. Math., 58 (1991), pp.~603--611.

\bibitem{elliott2010modeling}
{\sc C.~M. Elliott and B.~Stinner}, {\em Modeling and computation of two phase geometric biomembranes using surface finite elements}, J. Comput. Phys., 229 (2010), pp.~6585--6612.

\bibitem{epstein1987curve}
{\sc C.~L. Epstein and M.~Gage}, {\em The curve shortening flow}, in Wave Motion: Theory, Modelling, and Computation: Proceedings of a Conference in Honor of the 60th Birthday of Peter D. Lax, A.~J. Chorin and A.~J. Majda, eds., Springer US, New York, NY, 1987, pp.~15--59.

\bibitem{eyre1998unconditionally}
{\sc D.~J. Eyre}, {\em Unconditionally gradient stable time marching the {C}ahn-{H}illiard equation}, in Computational and mathematical models of microstructural evolution ({S}an {F}rancisco, {CA}, 1998), vol.~529 of Mater. Res. Soc. Sympos. Proc., MRS, Warrendale, PA, 1998, pp.~39--46.

\bibitem{garcke2025structure}
{\sc H.~Garcke, W.~Jiang, C.~Su, and G.~Zhang}, {\em Structure-preserving parametric finite element method for surface diffusion based on {L}agrange multiplier approaches}, SIAM J. Sci. Comput., 47 (2025), pp.~A1983--A2011.

\bibitem{helfrich1973elastic}
{\sc W.~Helfrich}, {\em Elastic properties of lipid bilayers: Theory and possible experiments}, Zeitschrift für Naturforschung C, 28 (1973), pp.~693--703.

\bibitem{huang2020highly}
{\sc F.~Huang, J.~Shen, and Z.~Yang}, {\em A highly efficient and accurate new scalar auxiliary variable approach for gradient flows}, SIAM J. Sci. Comput., 42 (2020), pp.~A2514--A2536.

\bibitem{huang2026weighted}
{\sc Q.-A. Huang, W.~Jiang, J.~Z. Yang, and C.~Yuan}, {\em A weighted scalar auxiliary variable method for solving gradient flows: bridging the nonlinear energy-based and {L}agrange multiplier approaches}, J. Sci. Comput., 106 (2026), pp.~Paper No. 16, 29.

\bibitem{kemmochi2025structure}
{\sc T.~Kemmochi, Y.~Miyatake, and K.~Sakakibara}, {\em Structure-preserving numerical methods for constrained gradient flows of planar closed curves with explicit tangential velocities}, Jpn. J. Ind. Appl. Math., 42 (2025), pp.~575--603.

\bibitem{kemmochi2022scalar}
{\sc T.~Kemmochi and S.~Sato}, {\em Scalar auxiliary variable approach for conservative/dissipative partial differential equations with unbounded energy functionals}, BIT, 62 (2022), pp.~903--930.

\bibitem{kimura1997numerical}
{\sc M.~Kimura}, {\em Numerical analysis of moving boundary problems using the boundary tracking method}, Japan Journal of Industrial and Applied Mathematics, 14 (1997), pp.~373--398.

\bibitem{elliott2017approximations}
{\sc C.~M.~Elliott and H.~Fritz}, {\em On approximations of the curve shortening flow and of the mean curvature flow based on the deturck trick}, IMA Journal of Numerical Analysis, 37 (2017), pp.~543--603.

\bibitem{mikula2001evolution}
{\sc K.~Mikula and D.~{\v{S}}ev{\v{c}}ovi{\v{c}}}, {\em Evolution of plane curves driven by a nonlinear function of curvature and anisotropy}, SIAM Journal on Applied Mathematics, 61 (2001), pp.~1473--1501.

\bibitem{mikula2004direct}
\leavevmode\vrule height 2pt depth -1.6pt width 23pt, {\em A direct method for solving an anisotropic mean curvature flow of plane curves with an external force}, Mathematical Methods in the Applied Sciences, 27 (2004), pp.~1545--1565.

\bibitem{mullins1956two-dimensional}
{\sc W.~W. Mullins}, {\em Two-dimensional motion of idealized grain boundaries}, J. Appl. Phys., 27 (1956), pp.~900--904.

\bibitem{mullins1957theory}
\leavevmode\vrule height 2pt depth -1.6pt width 23pt, {\em Theory of thermal grooving}, Journal of Applied Physics, 28 (1957), pp.~333--339.

\bibitem{sevcovic2011evolution}
{\sc D.~{\v{S}}ev{\v{c}}ovi{\v{c}} and S.~Yazaki}, {\em Evolution of plane curves with a curvature adjusted tangential velocity}, Japan Journal of Industrial and Applied Mathematics, 28 (2011), pp.~413--442.

\bibitem{shen2018scalar}
{\sc J.~Shen, J.~Xu, and J.~Yang}, {\em The scalar auxiliary variable ({SAV}) approach for gradient flows}, J. Comput. Phys., 353 (2018), pp.~407--416.

\bibitem{willmore1993riemannian}
{\sc T.~J. Willmore}, {\em Riemannian geometry}, Oxford Science Publications, The Clarendon Press, Oxford University Press, New York, 1993.

\bibitem{yang2017numerical}
{\sc X.~Yang, J.~Zhao, and Q.~Wang}, {\em Numerical approximations for the molecular beam epitaxial growth model based on the invariant energy quadratization method}, J. Comput. Phys., 333 (2017), pp.~104--127.

\bibitem{zhang2024new}
{\sc J.~Zhang and X.~Wang}, {\em A new scalar auxiliary variable approach for gradient flows}, 2024.

\end{thebibliography}

\end{document}